\documentclass{amsart}
\usepackage{amssymb}
\usepackage{amscd}
\usepackage[latin1]{inputenc}
\usepackage{amsmath}
\usepackage{amsfonts}
\usepackage{latexsym}
\usepackage{verbatim}
\usepackage{graphicx}
\usepackage{epsfig}
\newtheorem{theorem}{Theorem}[section]
\newtheorem{corollary}[theorem]{Corollary}
\newtheorem{lemma}[theorem]{Lemma}

\newtheorem*{Theorem1}{Theorem \ref{Lcombbd}}
\newtheorem*{Theorem2}{Theorem \ref{concat}}
\newtheorem*{Theorem3}{Theorem \ref{Gmeascor}}
\newtheorem*{Theorem4}{Theorem \ref{mainthm0}}
\newtheorem*{Theorem5}{Theorem \ref{mainthm2}}
\newtheorem*{Theorem5.5}{Theorem ~\ref{gibbsthm0}}
\newtheorem*{Theorem6}{Theorem \ref{gibbsthm}}
\newtheorem*{Theorem6.5}{Theorem \ref{hardbetacor}}
\newtheorem*{Theorem7}{Theorem \ref{bddthm}}
\newtheorem*{Theorem8}{Theorem \ref{bddpmthm}}

\theoremstyle{definition}
\newtheorem{definition}[theorem]{Definition}

\newtheorem{remark}[theorem]{Remark}

\begin{document}

\title[On subshifts with slow forbidden word growth]{On subshifts with slow forbidden word growth}

\begin{abstract}

In this work, we treat subshifts, defined in terms of an alphabet $\mathcal{A}$ and (usually infinite) forbidden list $\mathcal{F}$, where the number of $n$-letter words in $\mathcal{F}$ has ``slow growth rate'' in $n$. We show that such subshifts are well-behaved in several ways; for instance, they are boundedly supermultiplicative in the sense of \cite{baker} and they have unique measures of maximal entropy with the K-property and which satisfy Gibbs bounds on large (measure-theoretically) sets.

The main tool in our proofs is a more general result which states that bounded supermultiplicativity and a sort of measure-theoretic specification property together imply uniqueness of MME and our Gibbs bounds. 

We also show that some well-known classes of subshifts can be treated by our results, including the symbolic codings of $x \mapsto \alpha + \beta x$ (the so-called $\alpha$-$\beta$ shifts from \cite{hofbauerAB}) and the bounded density subshifts of \cite{stanley}. 

\end{abstract}

\date{}
\author{Ronnie Pavlov}
\address{Ronnie Pavlov\\
Department of Mathematics\\
University of Denver\\
2390 S. York St.\\
Denver, CO 80208}
\email{rpavlov@du.edu}
\urladdr{www.math.du.edu/$\sim$rpavlov/}
\thanks{The author gratefully acknowledges the support of a Simons Foundation Collaboration Grant.}
\renewcommand{\subjclassname}{MSC 2010}
%below are the definitions for some subject classification numbers
%22D40 Ergodic theory on groups 
%37A05 Measure-preserving transformations 
%37A15 General groups of measure preserving transformations 
%37A35 Entropy and other invariants, isomorphism, classification 
%37B10 symbolic dynamics
%37B40 topological entropy
%37B50 multi-dimensional shifts of finite type, tiling systems 
%37C40 smooth ergodic theory, invariant measures
%37C45 dimension theory of dynamical systems
%37C85 Dynamics of group actions other than Z and R, and foliations 
%37C99 smooth dynamical systems, general theory
%37D35 thermodynamic formalism, variational principles, equilibrium states

\maketitle

%\SetWatermarkScale{4}

\section{Introduction}\label{intro}

In this work, we study (one-dimensional) subshifts, which are symbolically defined topological dynamical systems. A subshift can be defined in terms of a finite set $\mathcal{A}$ (called the alphabet) and a set $\mathcal{F}$ of (finite) forbidden words on the alphabet; the induced subshift is the set of all biinfinite sequences on $\mathcal{A}$ which do not contain any word from $\mathcal{F}$. An extremely well-studied case occurs when $\mathcal{F}$ is finite. The induced subshift is then called a shift of finite type (or SFT), and the behavior of SFTs is in many senses very well understood. 

There are clearly only countably many SFTs, and so it is important to try to understand the general case where $\mathcal{F}$ is infinite as well. A plausible heuristic is that the longer a word $w$ is, the less of an effect forbidding $w$ has on the subshift. Therefore, it seems reasonable to expect that if $\mathcal{F}$ is `small' in the sense of having numbers of $n$-letter words which grow slowly in terms of $n$, then $X$ should be `close to' the class of SFTs and have correspondingly good behavior. However, there are very few results in the literature of this sort. Some notable related works include \cite{buzzi} (which used a hypothesis of slow growth of so-called one-sided constraint words to prove various properties, including finitely many ergodic MMEs), \cite{gurevic} (which used a hypothesis in terms of rapid convergence of entropies of SFT approximations of $X$ to prove uniqueness of MME), and \cite{miller} (which used a quantitative condition on the sets $\mathcal{F}_n$ of $n$-letter words in $\mathcal{F}$ to prove nonemptiness). In this work, we show that if $|\mathcal{F}_n|$ has slow enough growth compared to the alphabet $\mathcal{A}$, then $X$ has many properties similar to those of irreducible SFTs.

The main such property that we derive is uniqueness of the measure of maximal entropy (or MME), which we prove via a more general result which may be of independent interest. This result extends recent work (see, among other work, \cite{CT0}, \cite{CT1}, \cite{gapspec}) using weakened specification properties as hypotheses. Much like the work of Climenhaga and Thompson in \cite{CT0} and \cite{CT1}, we require the ability to combine only a restricted set of words, but rather than the prefix-center-suffix decomposition used in their work, we instead require that this restricted set be measure-theoretically large for MMEs.

\begin{Theorem4}%\label{mainthm0}
If $X$ is a subshift, there exists $C$ so that $|\mathcal{L}_n(X)| < Ce^{nh(X)}$ for all $n$, there exist $G \subset \mathcal{L}(X)$ and $R \in \mathbb{N}$ where for all $v, w \in G$, there exists $y \in \mathcal{A}^R$ so that $vyw \in \mathcal{L}(X)$, and for every ergodic MME $\mu$, there exist $\epsilon > 0$ and a syndetic $S$ so that $\mu(G_n) > \epsilon$ for every $n \in S$, then $X$ has a unique measure of maximal entropy.
\end{Theorem4}

We then prove that `small' $\mathcal{F}$ implies the three hypotheses of Theorem~\ref{mainthm0} (in fact for $R = 0$, meaning that words in $G$ are actually concatenable). We begin with the upper bound of the form $Ce^{nh(X)}$ on $|\mathcal{L}_n(X)|$, which is called bounded supermultiplicativity in \cite{baker}. 

\begin{Theorem1}
If there exists $\beta < 2h(X) - \log |\mathcal{A}|$ for which 
$\sum_{n = 1}^{\infty} n|\mathcal{F}_n| e^{-n \beta} < \frac{1}{36}$, then
$|\mathcal{L}_n(X)| < 4 e^{nh(X)}$ for all sufficiently large $n$.
\end{Theorem1}

The other two hypotheses of Theorem~\ref{mainthm0} require a set $G$; we define it to be the set of `good' words which neither begin nor end with more than one-third of a word in $\mathcal{F}$. (See Section~\ref{defg} for details.)

\begin{Theorem2}%\label{concat}
If there exists $c > |\mathcal{A}|^{-1}$ so that $\sum_{n = 1}^{\infty} |\mathcal{F}_n| c^n < \frac{|\mathcal{A}| c - 1}{2}$ and 
$\sum_{n = 1}^{\infty} n |\mathcal{F}_n| c^{n/3} < 1/4$, then for any $v,w \in G$, $vw \in \mathcal{L}(X)$.
\end{Theorem2}

\begin{Theorem3}%\label{Gmeascor}
If there exists $\alpha < h(X)$ so that $\sum_{n = 1}^{\infty} n^2 |\mathcal{F}_n| e^{-(n/3)\alpha} < 1 - e^{-\alpha}$, then for every ergodic MME $\mu$ on $X$, there exist $\epsilon > 0$ and a syndetic set $S$ so that $\mu(G_n) > \epsilon$ for all $n \in S$.
\end{Theorem3}

We can then use Theorems~\ref{mainthm0}, \ref{Lcombbd}, \ref{concat}, and \ref{Gmeascor} to prove uniqueness of MME, and even the Kolmogorov or K-property for that measure, via a simple quantitative condition on $\mathcal{F}$.

\begin{Theorem5}%\label{mainthm2}
If $\sum_{n = 1}^{\infty} n^2 |\mathcal{F}_n| (3/|\mathcal{A}|)^{n/3} < \frac{1}{36}$, then $X$ has a unique MME, which has the K-property.
\end{Theorem5}

We can also prove a restricted Gibbs property for $\mu$ under slightly stronger hypotheses. Again, we first prove a more general result, requiring a slightly stronger weakened specification property than Theorem~\ref{mainthm0}.

\begin{Theorem5.5}
If $X$ is a subshift with unique MME $\mu$, there exists $C$ so that $|\mathcal{L}_n(X)| < Ce^{nh(X)}$ for all $n$, there exist $G' \subset \mathcal{L}(X)$ and $R \in \mathbb{N}$ so that for all $u,v,w \in G'$, there exist
$y,z \in \mathcal{A}^R$ for which $uyvzw \in \mathcal{L}(X)$, and there also exist $\epsilon > 0$ and a syndetic $S$ so that $\mu(G'_n) > \epsilon$ for every $n \in S$, then $G'$ has the following Gibbs property: there exists $D$ so that for all $w \in G'$, 
\[
\mu([w]) \geq De^{-|w|h(X)}.
\]
\end{Theorem5.5}

Since this requires a stronger specification property than that of Theorem~\ref{mainthm0}, verifying this for general subshifts with `small' $\mathcal{F}$ requires a slightly different definition; we define $G'$ to be the set of words which neither begin nor end with more than one-fourth of a word in $\mathcal{F}$. 
(See Section~\ref{defg} for details.) Via simple modifications to the proofs of Theorems~\ref{concat} and \ref{Gmeascor} adapted to $G'$ (and again using $R = 0$), we can prove the following.

\begin{Theorem6}%\label{gibbsthm}
If $\sum_{n = 1}^{\infty} n^2 |\mathcal{F}_n| (3/|\mathcal{A}|)^{n/4} < \frac{1}{36}$ and $\mu$ is the unique MME on the induced subshift $X$ (guaranteed by Theorem~\ref{mainthm2}), then there exist $\epsilon > 0$ and a syndetic set $S$ so that $\mu(G'_n) > \epsilon$ for all $n \in S$, and there exist constants 
$D, D'$ so that for all $w \in G'$,
\[
De^{-|w|h(X)} \leq \mu([w]) \leq D'e^{-|w|h(X)}.
\]
\end{Theorem6}

\begin{remark}
We would like to briefly compare and contrast Theorems~\ref{mainthm2} and \ref{gibbsthm} with the main results from \cite{buzzi}, which are some of the most similar that we are aware of. In \cite{buzzi}, the author considers so-called minimal left (or right) constraints, which are words $w$ with the property that there is a word $v$ which cannot legally follow (or precede) $w$, but which can when the first (last) letter of $w$ is removed. Under the hypothesis that the exponential growth rate
\[
\limsup \frac{1}{n} \log |\mathcal{C}_n(X)|
\]
(here $\mathcal{C}_n(X)$ could represent the $n$-letter minimal left or right constraints) is less than $h(X)$, many useful properties of $X$ are derived, including finitely many ergodic MMEs, each of which is Bernoulli up to some periodicity. 

To summarize: our conclusions are similar to those of \cite{buzzi}, but some are weaker (e.g. Kolmogorov vs. Bernoulli) and some are stronger (e.g. uniqueness of MME vs. finitely many ergodic MME and our Gibbs bounds). Our hypotheses are similar in that they involve bounding the size of certain sets from above, but our assumptions are `finer' in the sense that they require explicit upper bounds on infinite series. We also require different bounds, in that our sets must grow with rate less than $|\mathcal{A}|^{n/3}$ rather than rate less than $e^{nh(X)}$. Finally, there is always a forbidden list $\mathcal{F}$ (the so-called first offenders/minimal forbidden words) where $|\mathcal{F}_{n+1}| \leq n|\mathcal{C}_{n}(X)|$ for every $n$ (see Lemma 15 from \cite{buzzi}), and so the sets we consider essentially grow at most at fast as those considered in \cite{buzzi}. They may in fact be much smaller, since a subshift $X$ may have other forbidden lists with much slower growth (see examples on p. 389 of \cite{buzzi}.) 

\end{remark}

Finally, we apply our techniques to some well-known classes of subshifts. The first is the $\alpha$-$\beta$ shifts of 
\cite{hofbauerAB} (denoted by $X_{\alpha, \beta}$). These shifts have corresponding sets $\mathcal{F}_{\alpha, \beta}$ of forbidden words and $G'_{\alpha,\beta}$ of words which do not begin or end with more than one-fourth of a forbidden word (see Section~\ref{apps} for details).

\begin{Theorem6.5}%\label{hardbetacor}
For every $\ell$, there exists $\epsilon > 0$ so that if $\alpha < \epsilon$, $\beta > \ell - \epsilon$, and $\alpha + \beta < \ell$, then $X_{\alpha, \beta}$ satisfies the conclusions of Theorems~\ref{mainthm2} and \ref{gibbsthm} (with $G' = G'_{\alpha, \beta}$).
\end{Theorem6.5} 

Our other application is to the bounded density shifts (denoted by $X_{k,h}$) from \cite{stanley}, as well as a natural generalization which we call signed bounded density shifts (denoted by $X^{\pm}_{k,h}$). These shifts have corresponding sets 
$\mathcal{F}_{k,h}$ (and $\mathcal{F}^{\pm}_{k,h}$) of forbidden words and sets $G_{k,h}$ (and $G^{\pm}_{k,h}$) of words which do not begin or end with more than one-third of a forbidden word (see Section~\ref{apps} for details). %In the following results, $\alpha$ refers to the `gradient' $\lim_{n \rightarrow \infty} h(n)/n$ of $h$ (this limit exists because $h$ is always assumed subadditive).

\begin{Theorem7}%\label{gibbsthm}
If $k > 9e$, $h(n) = nk$ for $n < 11$ and $h(n) > nk(1 - \frac{1}{9e})$ for all $n$, then $X_{k,h}$ satisfies the conclusions of Theorems~\ref{mainthm2} and \ref{gibbsthm} (with $G' = G_{k,h}$).
\end{Theorem7}

\begin{Theorem8}%\label{gibbsthm}
If $k > 9e$, $h(n) = nk$ for $n < 11$ and $h(n) > nk(1 - \frac{1}{9e})$ for all $n$, then $X^{\pm}_{k,h}$ satisfies the conclusions of Theorems~\ref{mainthm2} and \ref{gibbsthm} (with $G' = G^{\pm}_{k,h}$).
\end{Theorem8}

Section~\ref{defs} contains definitions and preliminary results which we will need. Section~\ref{unique} contains the proofs of Theorems~\ref{mainthm0} and \ref{gibbsthm0}, Section~\ref{ub} contains the proofs of Theorem~\ref{Lcombbd} and related results, Section~\ref{defg} contains the proofs of Theorem~\ref{concat} and related results, Section~\ref{bounding} contains the proof of Theorem~\ref{Gmeascor} and related results, and Section~\ref{final} contains the proofs of Theorem~\ref{mainthm2} and 
\ref{gibbsthm}. Finally, Section~\ref{apps} contains proofs of Theorems~\ref{hardbetacor}, \ref{bddthm}, and \ref{bddpmthm}.

\section{Definitions and preliminaries}\label{defs}

We begin with some basic definitions from symbolic dynamics; for a more in-depth introduction, see \cite{LM}.

\begin{definition}
For any finite alphabet $\mathcal{A}$, the \textbf{full shift} over $\mathcal{A}$ is the set $\mathcal{A}^{\mathbb{Z}} = \{\ldots x_{-1} x_0 x_1 \ldots \ : \ x_i \in \mathcal{A}\}$, which is viewed as a compact topological space with the (discrete) product topology.
\end{definition}

\begin{definition}
A \textbf{word} over $\mathcal{A}$ is any member $w = w_1 \ldots w_n$ of $\mathcal{A}^n$ for some $n \in \mathbb{N}$; $n$ is called the 
\textbf{length} of $w$ and denoted by $|w|$. The set of all words over $\mathcal{A}$ is denoted by $\mathcal{A}^*$. 
\end{definition}

For any set of words $S \subseteq \mathcal{A}^*$, we make the notation $S_n := S \cap \mathcal{A}^n$.

\begin{definition}
A word $w$ is a \textbf{subword} of a longer word or biinfinite sequence $x$ if there exists $n$ so that $x_n \ldots x_{n + |w| - 1} = w_1 \ldots w_{|w|}$.
We say that $w$ is a \textbf{prefix} of a longer word $v$ if $v_1 \ldots v_{|w|} = w_1 \ldots w_{|w|}$, and that $w$ is a \textbf{suffix} of 
$v$ if $v_{|v| - |w| + 1} \ldots v_{|v|} = w_1 \ldots w_{|w|}$. 
\end{definition}

We will colloquially refer to words as being the same if they are shifts of each other, i.e. in the above definition, we would usually say
$x_n \ldots x_{n + |w| - 1} = w$, even though these words technically have different domains when viewed as functions.

\begin{definition}
The \textbf{shift action} $\sigma$ is the automorphism of a full shift defined by $(\sigma x)_n = x_{n+1}$ for $n \in \mathbb{Z}$. 
\end{definition}

\begin{definition}
A \textbf{subshift} is a closed subset of a full shift $\mathcal{A}^{\mathbb{Z}}$ which is invariant under $\sigma$.
\end{definition}

Any subshift $X$ is a compact space with the induced topology from $\mathcal{A}^{\mathbb{Z}}$, and so $(X,\sigma)$ is a topological dynamical system. Subshifts can be equivalently defined in terms of a set of `forbidden' words.

\begin{definition}
For any alphabet $\mathcal{A}$ and $\mathcal{F} \subset \mathcal{A}^*$, define the \textbf{subshift induced by $\mathcal{A}$ and $\mathcal{F}$} as
\[
X = X(\mathcal{A}, \mathcal{F}) := \{x \in \mathcal{A}^{\mathbb{Z}} \ : \ x \textrm{ contains no subword in } \mathcal{F}\}.
\]
\end{definition}

It is well known that any $X(\mathcal{A}, \mathcal{F})$ is a subshift, and that all subshifts are representable in this way.

\begin{definition}
The \textbf{language} of a subshift $X$, denoted by $\mathcal{L}(X)$, is the set of all words which appear in points of $X$. For any $n \in \mathbb{Z}$, $\mathcal{L}_n(X) := 
\mathcal{L}(X) \cap \mathcal{A}^n$, the set of words in the language of $X$ with length $n$. 
\end{definition}

\begin{definition}
For any subshift and word $w \in \mathcal{L}_n(X)$, the \textbf{cylinder set} $[w]$ is the set of all $x \in X$ with $x_1 x_2 \ldots x_n = w$.  
\end{definition}

\begin{definition}\label{topent}
The \textbf{topological entropy} of a subshift $X$ is
\[
h(X) := \lim_{n \rightarrow \infty} \frac{1}{n} \ln |\mathcal{L}_n(X)|.
\]
\end{definition}

The existence of this limit follows from a standard subadditivity argument, which also implies that the limit can be replaced by an infimum, i.e. for any $n$, 
$h(X) \leq \frac{1}{n} \ln |\mathcal{L}_n(X)|$, or equivalently
\begin{equation}\label{basic}
|\mathcal{L}_n(X)| \geq e^{nh(X)}.
\end{equation}

\begin{definition}
For any subshift $X \subset \mathcal{A}^{\mathbb{Z}}$ and any $k \in \mathbb{N}$, the \textbf{$k$th higher power shift} associated to $X$, denoted $X^k$, is a subshift with alphabet $\mathcal{L}_k(X)$ defined by the following rule: $y \in (\mathcal{L}_k(X))^{\mathbb{Z}}$ is an element of $X^k$ if and only if the sequence 
$x \in A^{\mathbb{Z}}$ defined by concatenating the `letters' of $y$ is in $X$. 
(Formally, $\forall n \in \mathbb{Z}$, the $n$th letter of $x$ is defined to be the $(n \pmod k)$th letter of $y_{\lfloor n/k \rfloor}$.)
\end{definition}

It is well-known that the dynamical systems $(X^k, \sigma)$ and $(X, \sigma^k)$ are topologically conjugate, and that $h(X^k) = kh(X)$.

\

We also need some definitions from measure-theoretic dynamics; all measures considered in this paper will be Borel probability measures on a full shift $\mathcal{A}^{\mathbb{Z}}$.

\begin{definition}
A measure $\mu$ on $\mathcal{A}^{\mathbb{Z}}$ is {\bf ergodic} if any measurable set $C$ which is shift-invariant, meaning $\mu(C \triangle \sigma C) = 0$, has measure $0$ or $1$. 
\end{definition}

Not all $\sigma$-invariant measures are ergodic, but a well-known result called the ergodic decomposition shows that any non-ergodic measure can be written as a ``weighted average'' (formally, an integral) of ergodic measures; see 
Chapter 6 of \cite{walters} for more information. %One application of ergodic measures comes from Birkhoff's pointwise ergodic theorem, stated here only for the case of ergodic $\mu$ on a full shift $\mathcal{A}^{\mathbb{Z}}$.

One of the main strengths of ergodic measures is Birkhoff's pointwise ergodic theorem.

\begin{theorem}{\rm (Birkhoff's pointwise ergodic theorem)}\label{birkhoff}
For any ergodic measure $\mu$ on a subshift $X$ and any $f \in L^1(\mathcal{A}^{\mathbb{Z}},\mu)$,
\[
\lim_{n \rightarrow \infty} \frac{1}{n} \sum_{i=0}^{n-1} f(\sigma^i x) \underset{\mu{\rm -a.e.}}{\rightarrow} \int f \ d\mu.
\]
\end{theorem}

There is also a measure-theoretic version of entropy. 

\begin{definition}\label{measent}
For any $\sigma$-invariant measure $\mu$ on a full shift $\mathcal{A}^{\mathbb{Z}}$, the \textbf{measure-theoretic entropy} of $\mu$ is
\[
h(\mu) := \lim_{n \rightarrow \infty} \frac{-1}{n}  \sum_{w \in \mathcal{A}^n} \mu([w]) \ln \mu([w]),
\]
where terms with $\mu([w]) = 0$ are omitted from the sum.
\end{definition}

Ergodicity and measure-theoretic entropy are connected by the classical Shannon-McMillan-Brieman theorem.

\begin{theorem}{\rm (Shannon-McMillan-Brieman theorem)}
For an ergodic measure $\mu$ on a full shift $\mathcal{A}^{\mathbb{Z}}$,
\[
\lim_{n \rightarrow \infty} \frac{- \log \mu([x_1 \ldots x_n])}{n}  \rightarrow h(\mu)
\]
both $\mu$-a.e. and in $L^1(\mu)$. 
\end{theorem}

Measure-theoretic and topological entropies are connected by the famous Variational Principle.

\begin{theorem}{\rm (Variational Principle)}\label{var}
For a subshift $X$, 
$h(X) = \sup_{\mu} h(\mu)$,
where the supremum is taken over all $\sigma$-invariant measures $\mu$ with $\mu(X) = 1$. 
\end{theorem}

\begin{definition}
For any subshift $X$, a \textbf{measure of maximal entropy} on $X$ is a measure $\mu$ with support contained in $X$ for which $h(\mu) = h(X)$.
\end{definition}

It is well-known that every subshift has at least one measure of maximal entropy, and the ergodic decomposition and affineness of the entropy map (see Theorem 8.7(ii) in \cite{walters}) imply that if $X$ has multiple measures of maximal entropy, then it has multiple ergodic measures of maximal entropy.

\begin{definition}
A measure $\mu$ on a subshift $X$ is said to be \textbf{Kolmogorov} or have the \textbf{K-property} if there exists a $\sigma$-algebra $\mathcal{K}$ contained in $\mathcal{B}(X)$, the $\sigma$-algebra generated by all shifts of cylinder sets of words in $\mathcal{L}(X)$, with the following properties:\\

\noindent
$\bullet$ $\mathcal{K} \subset \sigma \mathcal{K}$

\noindent
$\bullet$ $\bigvee_{n = 0}^{\infty} \sigma^n \mathcal{K} = \mathcal{B}$ up to equivalence of sets with $\mu$-null symmetric difference

\noindent
$\bullet$ $\bigcap_{n = 0}^{\infty} \sigma^n \mathcal{K} = \{\varnothing, X\}$ up to equivalence of sets with $\mu$-null symmetric difference

\end{definition}

Among other properties, Kolmogorov $\mu$ have very strong mixing properties, and every non-trivial measure-theoretic factor has positive entropy. (See \cite{walters} for more information.)

Finally, we need some basic definitions about sets of natural numbers.

\begin{definition}
A set $S \subseteq \mathbb{N}$ is \textbf{syndetic} if there exists $N$ so that $\{k, \ldots, k + N - 1\} \cap S \neq \varnothing$ for every $k$; we then say that $N$ is an upper bound on the gaps of $S$.
\end{definition}

\begin{definition}
The \textbf{lower density} of a set $S \subseteq \mathbb{N}$ is
\[
\liminf_{n \rightarrow \infty} \frac{|S \cap \{1, \ldots, n\}|}{n}.
\]
\end{definition}

We note the easy fact that if $S$ has lower density greater than $\frac{1}{2}$, then $S + S = \{s_1 + s_2 \ : \ s_1, s_2 \in S\}$ is cofinite. Indeed, for large $n$, $|S \cap \{1, \ldots, n\}| > n/2$, which implies that $S$ contains either $n/2$ or some pair $k, n-k$ and so that $n \in S + S$.

\

Finally, we need two preliminary results for our proofs. The first is an argument of Miller from \cite{miller}, where he showed that under certain quantitative assumptions on $\mathcal{A}$ and $\mathcal{F}$, the induced subshift $X(\mathcal{A}, \mathcal{F})$ is nonempty. His results applied to one-sided subshifts (i.e. indexed by $\mathbb{N}$ rather than $\mathbb{Z}$). We will eventually need two different adaptations of his original argument, which we present here for context.

\begin{theorem}\label{ogmiller}{\rm (\cite{miller})}
If $\mathcal{F} \subset \mathcal{A}^*$ and there exists $c > |\mathcal{A}|^{-1}$ so that
\[
\sum_{v \in \mathcal{F}} c^{|v|} \leq |\mathcal{A}|c - 1, 
\]
then the one-sided shift $X(\mathcal{A}, \mathcal{F}) \subset \mathcal{A}^{\mathbb{N}}$ is nonempty.
\end{theorem}

\begin{proof}

Define, for every $w \in \mathcal{A}^*$, the weight function
\[
f_c(w) := \sum_{v \in \mathcal{F}} \sum_{\substack{r \in \mathcal{A}^*: |r| < |v|,\\ wr \textrm{ ends with } v}} c^{|r|}.
\]

It is clear that $f_c(\varnothing) = 0$, where $\varnothing$ is the empty word. Now, we note that for any $w$,
\begin{multline*}
\sum_{a \in \mathcal{A}} f_c(wa) = \sum_{v \in \mathcal{F}, a \in \mathcal{A}} \sum_{\substack{r \in \mathcal{A}^*: |r| < |v|,\\ war \textrm{ ends with } v}} c^{|r|}
= \sum_{ar \in \mathcal{F}} c^{|ar| - 1} + 
\sum_{v \in \mathcal{F}} \sum_{\substack{ar \in \mathcal{A}^*: |ar| < |v|,\\ w(ar) \textrm{ ends with } v}} c^{|ar| - 1}\\
= \frac{1}{c} \left(\sum_{v \in \mathcal{F}} c^{|v|} + f_c(w)\right).
\end{multline*}

If $f_c(w) < 1$, then $\sum_{a \in \mathcal{A}} f_c(wa) = \frac{1}{c} \left(\sum_{v \in \mathcal{F}} c^{|v|} + f_c(w)\right) < \frac{1}{c}(|\mathcal{A}|c - 1 + 1) = |\mathcal{A}|$.
This implies that there exists at least one $a \in \mathcal{A}$ so that $f_c(wa) < 1$. However, now by induction we can begin with $\varnothing$, which has weight $0 < 1$, and inductively add letters, creating a sequence $a_1 a_2 \ldots$ containing no words from $\mathcal{F}$. This sequence is in $X(\mathcal{A}, \mathcal{F})$ by definition.

\end{proof}

We will also need a combinatorial tool known as the Pliss Lemma (see \cite{pliss}).

\begin{lemma}\label{plisslem}
If $(a_n)$ is a sequence satisfying $0 \leq a_n \leq A$ for all $n$, $\alpha := \liminf (a_1 + \ldots + a_n)/n$, and $\beta < \alpha$, then the set
\[
\{n \ : \ \forall 0 \leq k < n, (a_{k+1} + \ldots + a_n)/(n-k) \geq \beta\}
\]
has lower density at least $\frac{\alpha - \beta}{A - \beta}$.

\end{lemma}

\section{General arguments for uniqueness of MME and Gibbs bounds}\label{unique}

In this section, we describe a general theorem which implies uniqueness of MME, which is an extension of the main result of \cite{gapspec} using a new measure-theoretic specification property. We first need the following elementary lemma from \cite{gapspec} (presented without proof). The result there was for general expansive systems and equilibrium states for nonzero potentials, but here we state a much simpler version for subshifts and measures of maximal entropy (corresponding to zero potential).

\begin{theorem}\label{specbd}
If $X$ is a subshift, $\mu$ is an ergodic MME on $X$, and $S \subset \mathcal{A}^n$, then
\[
|S| \geq \left( e^{nh(X)} \right)^{\frac{1}{\mu(S)}} |\mathcal{L}_n(X)|^{1 - \frac{1}{\mu(S)}} 2^{\frac{-1}{\mu(S)}}.
\]
\end{theorem}

We can now state our main uniqueness criterion.

\begin{theorem}\label{mainthm0}
If $X$ is a subshift, there exists $C$ so that $|\mathcal{L}_n(X)| < Ce^{nh(X)}$ for all $n$, there exist $G \subset \mathcal{L}(X)$ and $R \in \mathbb{N}$ where for all $v, w \in G$, there exists $y \in \mathcal{A}^R$ so that $vyw \in \mathcal{L}(X)$, and for every ergodic MME $\mu$, there exist $\epsilon > 0$ and a syndetic $S$ so that $\mu(G_n) > \epsilon$ for every $n \in S$, then $X$ has a unique measure of maximal entropy.
\end{theorem}

\begin{proof}

Our technique is extremely similar to that from \cite{gapspec}; that proof applied to arbitrary expansive systems and equilibrium states for nonzero potentials, and so here we reframe the argument for our simpler setting of a subshift and zero potential/MMEs. %In places we will omit details of some proofs which are identical to those from (REFERENCE), and refer the reader there for details.  

Assume that $X$ has the properties from the theorem, and assume for a contradiction that $X$ has unequal ergodic MMEs $\mu$ and $\nu$. By assumption, there exist $\delta > 0$ and syndetic sets $A, B$ so that $\mu(G_n) > \delta$ for all $n \in A$ and $\nu(G_n) > \delta$ for all $n \in B$. Unequal ergodic measures are mutually singular, and so there exists a clopen set $T$ so that $\mu(T), \nu(T^c) < \delta/5$. Choose $N$ large enough that it is greater than $R$, is an upper bound on the gaps in $A$ and $B$, and so that $T$ can be written as a union of cylinder sets of length $N$. 

Then, as was done in \cite{gapspec}, we can use the maximal ergodic theorem (specifically, Corollary 2.12 from \cite{gapspec} for $f = \chi_T$ and 
$\lambda = 2/5$), to define, for every $n$, $W_n \subset \mathcal{L}_n(X)$ with $\mu(W_n) > 1 - \delta/2$ so that every prefix of every $w \in W_n$ has number of subwords in $T$ less than $2/5$ its length. Similarly (by using $f = \chi_{T^c}$, $\lambda = 2/5$, and $\sigma^{-1}$ rather than $\sigma$), for every $n$, we define $V_n \subset \mathcal{L}_n(X)$ with $\nu(V_n) > 1 - \delta/2$ where every suffix of every $v \in V_n$ has number of subwords in $T^c$ less than $2/5$ its length. 

For every $n$, define $V'_n = V_n \cap G_n$ and $W'_n = W_n \cap G_n$; then $\mu(W'_n) > \delta/2$ for $n \in A$ and $\nu(V'_n) > \delta/2$ for $n \in B$. By Theorem~\ref{specbd} and the assumed upper bound on $\mathcal{L}_n(X)$, there is a constant $E$ so that $|W'_n| > E e^{nh(X)}$ for $n \in A$ and 
$|V'_n| > E e^{nh(X)}$ for $n \in B$. Now, for large enough $n$, we will create too many words in $\mathcal{L}_{n+i}(X)$ for some $0 \leq i < 3N$ (achieving a contradiction) by using the assumed property of $G$ to combine words in $V'_j$ and $W'_k$ for various lengths $j,k$.

Specifically, for every $0 \leq i < n/6N$, by definition of $N$, there exist $s, t \in [0, N)$ so that $6iN + s \in A$ and $n - 6iN + t \in B$. This implies that there is a set $C \subset [0, n/6N)$ of $i$-values with $|C| \geq \frac{n}{6N^3}$ so that all $i \in C$ share the same $s,t$; we will from now on only refer to $i \in C$ and treat $s, t$ as independent of $i$. Now, for every $i \in C$ and every $v \in V'_{6Ni + s}, w \in W'_{n - 6Ni + t}$, $v$ and $w$ are in $G$, and so there exists $y \in \mathcal{A}^R$ so that 
$vyw \in \mathcal{L}_{n + s + t + R}(X)$ by assumption. Clearly, for fixed $i$, all pairs $(v,w)$ yield different words $vyw$. We claim that distinct values of $i$ will also always yield different words. To wit, choose any $i_1 < i_2$ and any $v_1 \in V'_{6Ni_1 + s}$, $v_2 \in V'_{6Ni_2 + s}$, $w_1 \in W'_{n - 6Ni_1 + t}$, and $w_2 \in 
W'_{n - 6Ni_2 + t}$, define $y_1, y_2 \in \mathcal{A}^R$ so that $v_1 y_1 w_1, v_2 y_2 w_2 \in \mathcal{L}(X)$, and suppose for a contradiction that $v_1 y_1 w_1 = v_2 y_2 w_2$. 

Then the word occupying locations $6Ni_1 + s + R + 1$ through $6Ni_2 + s$, call it $u$, is both a prefix of $w_1$ and a suffix of $v_2$. Then $u$ has number of subwords in $T$ less than $2/5$ its length and number of subwords in $T^c$ less than $2/5$ its length. However, $u$ contains more than $|u| - N$ $N$-letter subwords, which is greater than $(4/5) |u|$ since $|u| = 6N(i_2 - i_1) - R > 5N$, and so we have a contradiction. 

We now know that all $i \in C$ and pairs $v, w$ yield distinct words in $\mathcal{L}_{n+s+t+R}(X)$, and so
\[
|\mathcal{L}_{n + 3N}(X)| \geq |\mathcal{L}_{n + s + t + R}(X)| \geq \sum_{i \in C} |V'_{6iN}| |W'_{n - 6iN}| \geq e^{-3Nh(X)} (n/6N^3) E^2 e^{(n+3N)h(X)}.
\]
This clearly contradicts the assumed upper bound on $|\mathcal{L}_n(X)|$ for large enough $n$, and so our original assumption was wrong, and $X$ has a unique MME.

%The assumption in (ME REFERENCE) was a non-uniform specification property (the details are not crucial here) which allowed concatenation of words from the subshift given a gap in between small in comparison to the combined words. This property implied an upper bound of the form  $|\mathcal{L}_n(X)| < n^{\epsilon} e^{nh(X)}$ for any desired $\epsilon$, weaker than our bound from Theorem~\ref{Lcombbd}. Then, Theorem~\ref{specthm} was used to show that $|V_n|, |W_n|$ could be bounded from below by $n^{-\delta} e^{nh(X)}$ for $\delta$ close to $0$. The final contradiction was then obtained by (nearly) concatenating words from $V_i$ and $W_j$ for various $i,j$ to create words in $\mathcal{L}_n(X)$. The sets obtained for different $i,j$ were disjoint by the defining properties of $V_i$ and $W_j$, and so this created on the order of $n^{1 - 2\delta} e^{nh(X)}$ words in $\mathcal{L}_n(X)$, a contradiction. 

%We need change very little about this proof. Our upper bound on $\mathcal{L}_n(X)$ is even stronger than that from (ME REFERENCE). We do not have a specification property for all words in $\mathcal{L}(X)$, but for two words in $G$, we can literally concatenate with no gap in order to create a word in $\mathcal{L}_n(X)$ by Theorem~\ref{concat}. In addition, we know that $\mu(G)$ and $\nu(G)$ are bounded away from $0$ by Theorem~\ref{Gmeasbd}. 

\end{proof}

\begin{remark} 
This proof would go through with few changes even if words from $G$ were combinable with a gap which is not constant, but grows sublogarithmically as a function of the lengths of the combined words (this was the type of hypothesis originally used in \cite{gapspec}). However, we are aware of no simple examples which require this more complicated hypothesis along with reduction to the subset $G$.
\end{remark}

With a slightly stronger weakened specification hypothesis, we can also prove a Gibbs lower bound on $G$ for the unique MME.

\begin{theorem}\label{gibbsthm0}
If $X$ is a subshift with unique MME $\mu$, there exists $C$ so that $|\mathcal{L}_n(X)| < Ce^{nh(X)}$ for all $n$, there exist $G' \subset \mathcal{L}(X)$ and $R \in \mathbb{N}$ so that for all $u,v,w \in G'$, there exist
$y,z \in \mathcal{A}^R$ for which $uyvzw \in \mathcal{L}(X)$, and there also exist $\epsilon > 0$ and a syndetic $S$ so that $\mu(G'_n) > \epsilon$ for every $n \in S$, then $G'$ has the following Gibbs property: there exists $D$ so that for all $w \in G'$, 
\[
\mu([w]) \geq De^{-|w|h(X)}.
\]
\end{theorem}

\begin{proof}
By an argument of Walters \cite{walters}, we may explicitly construct an MME on $X$ as follows. For every $n$ and every $w \in \mathcal{L}_n(X)$, choose any
$x(w) \in [w]$. For each $n$, define 
\[
\nu_n := \frac{1}{n |\mathcal{L}_n(X)|} \sum_{w \in \mathcal{L}_n(X), 0 \leq i < n} \sigma^i \circ \delta_{x(w)}.
\]
Then, any limit point of the measures $\nu_n$ (in the usual weak-* topology) is an MME on $X$ (see Theorem 8.6 from \cite{walters}). Since we already know $\mu$ is the only MME on $X$, $\nu_n \rightarrow \mu$ weak-*. By Theorem~\ref{specbd}, the assumed upper bound on $|\mathcal{L}_n(X)|$, and the assumed fact that $\mu(G'_n)$ is bounded away from $0$ along a syndetic set $S$, we know that there exists $E > 0$ so that $|G'_n| > Ee^{-nh(X)}$ for all $n \in S$. Choose any $N$ which is both greater than $R$ and an upper bound on gaps of $S$.

Now, suppose $v \in G'$, and fix any $n > |v| + 3N$. Choose any $N + R \leq i < n - N - R - |v|$; we wish to give a lower bound on the number of $w \in \mathcal{L}_n(X)$ which contain $v$ starting at distance $i$ from the beginning of $w$. By definition of $N$, there exist $j \in (i - R - N, i - R] \cap S$ and $k \in (n - |v| - R - i - N, n - |v| - R - i] \cap S$. We know that for every 
$u \in G'_j$ and $w \in G'_k$, there exist $y,z \in \mathcal{A}^R$ so that $uyvzw \in \mathcal{L}_{j + |v| + 2R + k}(X)$.  This yields at least $|G'_j| |G'_k| \geq E^2 e^{(j + k)h(X)} \geq E^2 e^{(n - |v| - 4N) h(X)}$ words in $\mathcal{L}_{j + |v| + 2R + k}(X)$, and by extending each arbitrarily to the left by $i - j - R$ letters and to the right by $n - |v| - R - i - k$ letters, we get at least $E^2 e^{(n - |v| - 4N) h(X)}$ words
in $\mathcal{L}_n(X)$ containing $v$ starting at distance $i$ from the beginning. This implies that
\begin{multline*}
\nu_n([v]) = \frac{1}{n |\mathcal{L}_n(X)|} \sum_{w \in \mathcal{L}_n(X), 0 \leq i < n} \sigma^i \circ \delta_{x(w)} 
\\ \geq \frac{1}{n |\mathcal{L}_n(X)|} \sum_{i = N+R}^{n - N - |v| - R - 1} E^2 e^{(n - |v| - 4N) h(X)}
= \frac{n - |v| - 2N - 2R}{n} \cdot \frac{E^2 e^{(n - |v| - 4N) h(X)}}{|\mathcal{L}_n(X)|}.
\end{multline*}

By the assumed upper bound on $|\mathcal{L}_n(X)|$, $\liminf \nu_n([v]) \geq \frac{E^2}{Ce^{4Nh(X)}} e^{-|v|h(X)}$, and so
$\mu([v]) \geq \frac{E^2}{Ce^{4Nh(X)}} e^{-|v|h(X)}$, completing the proof. 

\end{proof}

We recall that with even weaker hypotheses, one can prove a similar upper Gibbs bound for every word in $\mathcal{L}(X)$; this essentially appears
in \cite{CT0} as Lemma 5.12, but we restate here to make clear the hypotheses that are required.

\begin{theorem}\label{easygibbs}
If $X$ is a subshift with a unique MME $\mu$ and there exists $C$ so that $|\mathcal{L}_n(X)| \leq Ce^{nh(X)}$ for all $n$, then there exists a constant $D'$ so that for all $w \in \mathcal{L}(X)$,
\[
\mu([w]) \leq D'e^{-|w|h(X)}.
\]
\end{theorem}

\begin{proof}
Just as in the proof of Theorem~\ref{gibbsthm0}, we know that the unique MME $\mu$ is the weak-* limit of the measures 
\[
\nu_n := \frac{1}{n |\mathcal{L}_n(X)|} \sum_{w \in \mathcal{L}_n(X), 0 \leq i < n} \sigma^i \circ \delta_{x(w)}.
\]
Choose any $v \in \mathcal{L}(X)$, and any $n > |v|$. Then, for every $0 \leq i < n - |v|$, we can easily give an upper bound on the number of $w \in \mathcal{L}_n(X)$ which contain $v$ starting at distance $i$ from the beginning of $w$; it is obviously less than or equal to 
$|\mathcal{L}_i(X)| |\mathcal{L}_{n - i - |v|}(X)|$, which is in turn less than or equal to $C^2 e^{(n - |v|)h(X)}$ by assumption.  For other values of $i$, we make the trivial observation that this number is not more than $|\mathcal{L}_n(X)|$. Therefore,
\begin{multline*}
\nu_n([v]) = \frac{1}{n |\mathcal{L}_n(X)|} \sum_{w \in \mathcal{L}_n(X), 0 \leq i < n} \sigma^i \circ \delta_{x(w)} 
\\ \leq \frac{1}{n |\mathcal{L}_n(X)|} \left(\sum_{i = 0}^{n - |v|} C^2 e^{(n - |v|) h(X)} + \sum_{i = n - |v| + 1}^{n - 1} |\mathcal{L}_n(X)|\right)
\leq C^2 \frac{e^{(n-|v|) h(X)}}{|\mathcal{L}_n(X)|} + \frac{|v|}{n}.
\end{multline*}

By (\ref{basic}), $|\mathcal{L}_n(X)| \geq e^{nh(X)}$ for all $n$, and so $\limsup \nu_n([v]) \leq C^2 e^{-|v|h(X)}$. This implies that $\mu([v]) \leq C^2 e^{-|v|h(X)}$ and completes the proof. 

\end{proof}

The remainder of this paper will be devoted to showing that if our forbidden list $\mathcal{F}$ is `small' (in the sense of slow growth rate of 
$|\mathcal{F}_n|$), then $X$ satisfies the hypotheses of Theorems~\ref{mainthm0}, \ref{gibbsthm0}, and \ref{easygibbs}. We in fact will always be able to verify these hypotheses for $R = 0$, meaning that the sets $G$/$G'$ are actually concatenable rather than just being combinable with constant gap.

\section{Upper bounds on $|\mathcal{L}_n(X)|$}\label{ub}

From now on, we consider subshifts of the form $X(\mathcal{A}, \mathcal{F})$, and will assume that $X$ refers to this subshift unless explicitly stated otherwise.

In this section, we will show that `small' $\mathcal{F}$ implies an upper bound of the form $|\mathcal{L}_n(X)| < Ce^{nh(X)}$. We first adapt Miller's proof of 
Theorem~\ref{ogmiller} to yield a simple lower bound on entropy.

\begin{theorem}\label{millerent}
If there exists $c > |\mathcal{A}|^{-1}$ and $k < |\mathcal{A}|$ so that $\sum_{v \in \mathcal{F}} c^{|v|} < c(|\mathcal{A}| - k + 1) - 1$, then 
$h(X) \geq \log k$.
\end{theorem}

\begin{proof}
The proof is extremely similar to that of Theorem~\ref{ogmiller}. We simply recall that for every $w \in \mathcal{A}^*$, $\sum_{a \in \mathcal{A}} f_c(wa) = 
c^{-1}\left(\sum_{v \in \mathcal{F}} c^{|v|} + f_c(w)\right)$, and so under the hypotheses of the theorem, 
$f_c(w) < 1 \Longrightarrow \sum_{a \in \mathcal{A}} f_c(wa) < |\mathcal{A}| - k + 1$, implying that there are at least $k$ letters $a$ so that $f_c(wa) < 1$. 
If we define $X'$ to be the one-sided subshift induced by $\mathcal{A}$ and $\mathcal{F}$, then since $f_c(\varnothing) = 0 < 1$, we can prove by induction that 
$|\mathcal{L}_n(X')| \geq k^n$ for all $n$, and so that $h(X') \geq \log k$. We need to consider the two-sided subshift $X$ induced by $\mathcal{A}$ and $\mathcal{F}$ instead, but this is well-known to have the same entropy (it is the so-called natural extension of $X'$), which completes the proof.

\end{proof}

We will also need a simple corollary from the Pliss lemma to derive a technical result about the behavior of ratios $\frac{|\mathcal{L}_n(X)|}{|\mathcal{L}_{n-k}(X)|}$.

\begin{theorem}\label{plissthm}
%If $\sum_{n=1}^{\infty} |\mathcal{F}_n| (3/|\mathcal{A}|)^n < \frac{1}{5}$, then 
For any subshift $X$ and $\beta < 2h(X) - \log |\mathcal{A}|$, the set
\[
S := \left\{n \ : \ \forall k < n, \frac{|\mathcal{L}_n(X)|}{|\mathcal{L}_{n-k}(X)|} \geq e^{k \beta}\right\}
\]
has lower density greater than $\frac{1}{2}$.
\end{theorem}

\begin{proof}
This is an immediate corollary of the Pliss lemma with $a_n = \log |\mathcal{L}_n(X)| - \log |\mathcal{L}_{n-1}(X)|$, $\alpha = h(X)$, $A = \log |\mathcal{A}|$; 
note that $a_{n-k+1} + \ldots + a_n = \log \left(\frac{|\mathcal{L}_n(X)|}{|\mathcal{L}_{n-k}(X)|}\right)$, and that $\frac{\alpha - \beta}{A - \beta} > \frac{1}{2}$ since
$\beta < 2\alpha - A$.

\end{proof}

We are now prepared to prove the main result of this section.

\begin{theorem}\label{Lcombbd}
If there exists $\beta < 2h(X) - \log |\mathcal{A}|$ for which 
$\sum_{n = 1}^{\infty} n|\mathcal{F}_n| e^{-n \beta} < \frac{1}{36}$, then
$|\mathcal{L}_n(X)| < 4 e^{nh(X)}$ for all sufficiently large $n$.
\end{theorem}

\begin{proof}
Fix any $X$ and $\beta$ as in the statement. We first use Theorem~\ref{plissthm} to define the set $S = \{n \ : \ \forall k < n, \frac{|\mathcal{L}_n(X)|}{|\mathcal{L}_{n-k}(X)|} \geq
 e^{k\beta}\}$ of lower density greater than $\frac{1}{2}$. Fix any $n \in S$.

We will view letters of $\mathcal{L}_n(X)$ as an alphabet (for this proof, reference to a `letter' always means an element of $\mathcal{L}_n(X)$), and define a new set of forbidden words $\mathcal{F}^{(n)} \subset (\mathcal{L}_n(X))^*$. We say that $v \in \mathcal{F}^{(n)}$ if the concatenation of the `letters' $v_1, \ldots, v_m \in \mathcal{L}_n(X)$, viewed as a word over $\mathcal{A}$, contains some $w \in \mathcal{F}$, and if neither the concatenation of $v_1, \ldots, v_{m-1}$ nor the concatenation of $v_2, \ldots, v_m$ contains $w$. It should be clear from this definition that the alphabet $\mathcal{L}_n(X)$ and forbidden list 
$\mathcal{F}^{(n)}$ induce the $n$th higher power subshift $X^n$.

We first wish to bound from above the number of words in $(\mathcal{F}^{(n)})_j$ for each $j$. It is obvious that $(\mathcal{F}^{(n)})_1$ is empty for all $n$ since our `letters' $\mathcal{L}_n(X)$ do not themselves contain words in $\mathcal{F}$. 
So, if $v \in \mathcal{F}^{(n)}$, then $j \geq 2$ and $v$ (viewed as a concatenation) contains some $w \in \mathcal{F}$. By the fact that $v$ has no proper subword containing $w$, we know that $w$ has length between $(j-2)n + 2$ and $jn$, and we can parametrize by the length of $w$, which we call $i$, and the distance from the beginning of $v$ to the first letter of $w$, which we call $k$. Then for any $i$, the value of $k$ can be anywhere from $(j - 1)n - i - 1$ to 
$n - 1$. Also, for each $k$, $v$ is determined by $w$, $k$, and the portions of the first and last `letters' of $v$ not contained in $w$; these have lengths 
$k$ and $jn - i - k$, which are both less than or equal to $n$. This implies that
\begin{equation}\label{onlycase}
|(\mathcal{F}^{(n)})_j| \leq \sum_{i = (j - 2)n + 2}^{jn} \sum_{k = (j - 1)n - i - 1}^{n-1} |\mathcal{F}_i| |\mathcal{L}_k(X)| |\mathcal{L}_{jn - i - k}(X)|.
\end{equation}

We now note that by definition of $S$, for each $k$,
\begin{multline}\label{onlycase2}
|\mathcal{L}_k(X)| |\mathcal{L}_{jn - i - k}(X)| = |\mathcal{L}_n(X)|^2 \frac{|\mathcal{L}_k(X)|}{|\mathcal{L}_n(X)|} \frac{|\mathcal{L}_{jn - i - k}(X)|}{|\mathcal{L}_n(X)|} \\
\leq |\mathcal{L}_n(X)^2| e^{(-n + k - n + (jn - i - k))\beta} 
\leq |\mathcal{L}_n(X)^2| e^{(-i + (j - 2)n)\beta}.
\end{multline}

Combining (\ref{onlycase}) and (\ref{onlycase2}) yields
\begin{equation}\label{onlycase3}
|(\mathcal{F}^{(n)})_j| \leq |\mathcal{L}_n(X)|^2 \sum_{i = (j - 2)n + 2}^{jn} i |\mathcal{F}_i| e^{(-i + (j - 2)n)\beta}.
%= |\mathcal{L}_n(X)|^2 \left(\frac{|A|}{4}\right)^{(j-2)n} \sum_{i = (j - 2)n + 2}^{jn} i |\mathcal{F}_i| \left(\frac{4}{|\mathcal{A}|}\right)^i.
\end{equation}

We now use Theorem~\ref{millerent} and (\ref{onlycase3}) to bound the entropy $h(X^n) = nh(X)$ of $X^n$ from below. In particular, we claim that if $n$ is sufficiently large (and in $S$), the hypotheses of Theorem~\ref{millerent} are satisfied for $c = \frac{3}{|\mathcal{L}_n(X)|}$, alphabet $\mathcal{L}_n(X)$, $k = \frac{|\mathcal{L}_n(X)|}{2}$, and forbidden list $\mathcal{F}^{(n)}$. The relevant inequality we must verify is then
\begin{equation}\label{millerineq}
\sum_{j = 2}^{\infty} c^j |(\mathcal{F}^{(n)})_j| < \frac{c |\mathcal{L}_n(X)|}{2} - 1 = \frac{1}{2}.
\end{equation}

We now just bound from above the left-hand side of (\ref{millerineq}) using (\ref{onlycase3}):
\begin{equation}\label{onlycase4}
\sum_{j = 2}^{\infty} c^j |(\mathcal{F}^{(n)})_j| \leq
\sum_{j = 2}^{\infty} \left(\frac{3}{|\mathcal{L}_n(X)|} \right)^j |\mathcal{L}_n(X)|^2 e^{(j-2)n\beta} \sum_{i = (j - 2)n + 2}^{jn} i |\mathcal{F}_i| e^{-i\beta}.
\end{equation}

We recall by (\ref{basic}) that $|\mathcal{L}_n(X)| \geq e^{nh(X)}$ for all $n$. Therefore, for every $j$, 
\begin{equation}\label{onlycase5}
\left(\frac{3}{|\mathcal{L}_n(X)|} \right)^j |\mathcal{L}_n(X)|^2 e^{(j-2)n \beta} = \frac{3^j e^{(j-2)n \beta}}{|\mathcal{L}_n(X)|^{j-2}}
< \frac{3^j e^{(j-2)n \beta}}{e^{(j-2)n h(X)}} = 9 \left( e^{n(\beta - h(X))} \right)^{j-2}.
\end{equation}
Since $\beta < 2h(X) - \log |\mathcal{A}| \leq h(X)$, for large enough $n$, $e^{n(\beta - h(X))} < 1/3$, meaning that the expression in (\ref{onlycase5}) is less than $9$ for all $k$.
Combining with (\ref{onlycase4}) yields, for sufficiently large $n \in S$,
\[
\sum_{j = 2}^{\infty} c^j |(\mathcal{F}^{(n)})_j| \leq 9 \sum_{j = 2}^{\infty} \sum_{i = (j - 2)n + 2}^{jn} i |\mathcal{F}_i| e^{-i\beta}
\leq 18 \sum_{i = 2}^{\infty} i |\mathcal{F}_i| e^{-i\beta},
\]
which is less than $1/2$ by assumption. We have then demonstrated (\ref{millerineq}), and so $h(X^n) = nh(X) > \log(|\mathcal{L}_n(X)|/2)$, implying $|\mathcal{L}_n(X)| < 2e^{nh(X)}$
for sufficiently large $n \in S$. Denote by $S'$ the cofinite subset of $S$ for which this inequality holds. %This inequality is clearly true for $n = 1$ as well by Theorem~\ref{entbound}.

Since the lower density of $S$ is greater than $\frac{1}{2}$, the same is true of $S'$, meaning that $S' + S'$ is cofinite. Therefore, for all sufficiently large $n$, we can write $n = s_1 + s_2$ for $s_1, s_2 \in S'$, and then $|\mathcal{L}_n(X)| \leq |\mathcal{L}_{s_1}(X)| |\mathcal{L}_{s_2}(X)| < 
2e^{s_1h(X)} 2e^{s_2h(X)} = 4e^{nh(X)}$.

\end{proof}

\section{Defining $G$ and $G'$ with concatenability properties}\label{defg}

Our $G$ and $G'$ will be defined as words which do not begin or end with subwords of words in $\mathcal{F}$ which are too long.

\begin{definition}\label{heavy}
A \textbf{heavy subword} is a subword of some $w \in \mathcal{F}$ whose length is at least $|w|/3$. Denote by $\mathcal{H}$ the set of heavy subwords.
\end{definition}

\begin{definition}
For any subshift $X$ with forbidden list $\mathcal{F}$, define $P$ to be the set of all words in $\mathcal{L}(X)$ which do not have any heavy subword as a prefix, define $S$ to be the set of all words in $\mathcal{L}(X)$ which do not have any heavy subword as a suffix, and define $G = P \cap S$. %Similarly define $\widetilde{P}$, $\widetilde{S}$, and $\widetilde{G}$ for fourth-heavy subwords.
\end{definition}

It is immediate that the concatenation of two words in $G$ cannot contain a forbidden word (such a forbidden word would contain a heavy subword in one of the concatenated words!). To prove that such a concatenation is actually in $\mathcal{L}(X)$, we will need a version of Miller's proof of Theorem~\ref{ogmiller} adapted to two-sided subshifts. We first define a two-sided analogue of his weight function.

\begin{definition}\label{2sidedg}
If $\mathcal{F} \subseteq \mathcal{A}^*$ and $c > 0$, for any $w \in \mathcal{A}^*$ we define
\[
g_c(w) := \sum_{v \in \mathcal{F}} \left[ \sum_{\substack{\ell \in \mathcal{A}^*: |\ell| < |v|, \\ \ell w \textrm{ begins with } v}} c^{|\ell|} + 
\sum_{\substack{r \in \mathcal{A}^*: |r| < |v|, \\ w r \textrm{ begins with } v}} c^{|r|} + 
2 \sum_{\substack{\ell,r \in \mathcal{A}^*: |\ell|, |r| > 0, \\ \ell w r = v}} c^{|\ell| + |r|} \right].
\]
\end{definition}

We make the trivial observation that if $w$ begins with, ends with, or is equal to some $v \in \mathcal{F}$, then $g_c(w) \geq 1$.

\begin{theorem}\label{2miller}
If there exists $c > |\mathcal{A}|^{-1}$ so that 
\[
\sum_{n = 1}^{\infty} |\mathcal{F}_n| c^n < \frac{|\mathcal{A}| c - 1}{2}, 
\]
then any $w$ containing no word from $\mathcal{F}$ and satisfying $g_c(w) < 1$ is in $\mathcal{L}(X)$.
\end{theorem}

\begin{proof}
Our proof is essentially just a more complicated version of the proof of Theorem~\ref{ogmiller} adapted to the two-sided setting. The fundamental claim we wish to verify is that for every $w$, 
\begin{equation}\label{2millerineq}
\sum_{a, b \in \mathcal{A}} g_c(awb) \leq 2|\mathcal{A}|c^{-1} \sum_{n = 1}^{\infty} |\mathcal{F}_n| c^n + |\mathcal{A}| c^{-1} g_c(w).
\end{equation}
Once this is verified, we would know that whenever $g_c(w) < 1$,
\[
\sum_{a, b \in \mathcal{A}} g_c(awb) < 2|\mathcal{A}|c^{-1} \frac{|\mathcal{A}|c - 1}{2} + |\mathcal{A}| c^{-1} = |\mathcal{A}|^2,
\]
implying that there exist $a, b \in \mathcal{A}$ for which $g_c(awb) < 1$. Continuing in this way yields sequences $(a_n)$ and $(b_n)$ so that
$g_c(a_n \ldots a_1 w b_1 \ldots b_n) < 1$ for all $n$, implying that $a_n \ldots a_1 w b_1 \ldots b_n$ does not begin with, end with, or equal any $v \in \mathcal{F}$ for every $n$. Then, the biinfinite sequence $\ldots a_2 a_1 w b_1 b_2 \ldots$ contains no word in $\mathcal{F}$, and so is in $X$, implying that $w \in \mathcal{L}(X)$. It remains to verify (\ref{2millerineq}).

Fix any $w$, $a$, and $b$. Then by Definition~\ref{2sidedg},
\begin{multline}\label{awb}
\sum_{a,b \in \mathcal{A}} g_c(awb) = \\
\sum_{v \in \mathcal{F}, a, b \in \mathcal{A}} \left[ \sum_{\substack{\ell \in \mathcal{A}^*: |\ell| < |v|, \\ \ell awb\textrm{ begins with } v}} c^{|\ell|} + 
\sum_{\substack{r \in \mathcal{A}^*: |r| < |v|, \\ awb r \textrm{ begins with } v}} c^{|r|} + 
2 \sum_{\substack{\ell,r \in \mathcal{A}^*: |\ell|, |r| > 0, \\ \ell awb r = v}} c^{|\ell| + |r|} \right].
\end{multline}
We will rewrite each of the three terms on the right-hand side of (\ref{awb}).
\begin{multline}\label{part1}\tag{2A}
\sum_{v \in \mathcal{F}, a, b \in \mathcal{A}} \sum_{\substack{\ell \in \mathcal{A}^*: |\ell| < |v|, \\ \ell awb\textrm{ begins with } v}} c^{|\ell|} = 
|\mathcal{A}| \sum_{v \in \mathcal{F}, a \in \mathcal{A}} \sum_{\substack{\ell \in \mathcal{A}^*: \\ \ell a = v}} c^{|\ell|} \\ + 
|\mathcal{A}| \sum_{v \in \mathcal{F}, a \in \mathcal{A}} \sum_{\substack{\ell \in \mathcal{A}^*: |\ell| < |v|, \ell a \neq v, \\ \ell aw\textrm{ begins with } v}} c^{|\ell|} + 
\sum_{v \in \mathcal{F}, a, b \in \mathcal{A}} \sum_{\substack{\ell \in \mathcal{A}^*: \\ \ell a w b = v}} c^{|\ell|}\\
= \underbrace{|\mathcal{A}| c^{-1} \sum_{\ell a \in \mathcal{F}} c^{|\ell a|}}_{(I)} + 
\underbrace{|\mathcal{A}| c^{-1} \sum_{v \in \mathcal{F}, a \in \mathcal{A}} \sum_{\substack{\ell a \in \mathcal{A}^*: |\ell a| < |v|, \\ (\ell a) w\textrm{ begins with } v}} c^{|\ell a|}}_{(II)} + 
\underbrace{c^{-2} \sum_{v \in \mathcal{F}} \sum_{\substack{\ell a\in \mathcal{A}^*, b \in \mathcal{A}: \\ (\ell a) w b = v}} c^{|\ell a| + |b|}}_{(III)}.
\end{multline}

\begin{multline}\label{part2}\tag{2B}
\sum_{v \in \mathcal{F}, a, b \in \mathcal{A}} \sum_{\substack{r \in \mathcal{A}^*: |r| < |v|, \\ awbr\textrm{ ends with } v}} c^{|r|}
= |\mathcal{A}| \sum_{v \in \mathcal{F}, b \in \mathcal{A}} \sum_{\substack{r \in \mathcal{A}^*: \\ br = v}} c^{|r|} \\ + 
|\mathcal{A}| \sum_{v \in \mathcal{F}, b \in \mathcal{A}} \sum_{\substack{\ell \in \mathcal{A}^*: |r| < |v|, br \neq v, a w br \neq v, \\ awbr\textrm{ end with } v}} c^{|r|} + 
\sum_{v \in \mathcal{F}, a, b \in \mathcal{A}} \sum_{\substack{r \in \mathcal{A}^*: \\ a w br = v}} c^{|r|} \\
= \underbrace{|\mathcal{A}| c^{-1} \sum_{b r \in \mathcal{F}} c^{|br|}}_{(IV)} + 
\underbrace{|\mathcal{A}| c^{-1} \sum_{v \in \mathcal{F}, b \in B} \sum_{\substack{rb \in \mathcal{A}^*: |rb| < |v|, \\ w (br) \textrm{ ends with } v}} c^{|rb|}}_{(V)} + 
\underbrace{c^{-2} \sum_{v \in \mathcal{F}} \sum_{\substack{a\in \mathcal{A}, br \in \mathcal{A}^*: \\ a w (br) = v}} c^{|a| + |br|}}_{(VI)}.
\end{multline}

\begin{equation}\label{part3}\tag{2C}
2 \sum_{v \in \mathcal{F}, a, b \in \mathcal{A}} \sum_{\substack{\ell,r \in \mathcal{A}^*: |\ell|, |r| > 0, \\ \ell awb r = v}} c^{|\ell| + |r|} 
= \underbrace{2 c^{-2} \sum_{v \in \mathcal{F}} \sum_{\substack{\ell a, br \in \mathcal{A}^*: |\ell a|, |br| > 1, \\ (\ell a) w (b r) = v}} c^{|\ell a| + |b r|}}_{(VII)}. 
\end{equation}

(In (I), (II), (IV), and (V), the factor of $|\mathcal{A}|$ comes from removing the summation over either $a$ or $b$ in a constant quantity.) 
We now bound the right-hand side of (\ref{2millerineq}) from below. 
\begin{multline*}
2|\mathcal{A}| c^{-1} \sum_{n = 1}^{\infty} |\mathcal{F}_n| c^n + |\mathcal{A}| c^{-1} g_c(w) = 
2|\mathcal{A}| c^{-1} \sum_{n = 1}^{\infty} |\mathcal{F}_n| c^n + 
|\mathcal{A}| c^{-1} \sum_{v \in \mathcal{F}} \sum_{\substack{\ell \in \mathcal{A}^*: |\ell| < |v|, \\ \ell w \textrm{ begins with } v}} c^{|\ell|} + \\
|\mathcal{A}| c^{-1} \sum_{v \in \mathcal{F}} \sum_{\substack{r \in \mathcal{A}^*: |r| < |v|, \\ w r \textrm{ begins with } v}} c^{|r|} + 
2 |\mathcal{A}| c^{-1} \sum_{v \in \mathcal{F}} \sum_{\substack{\ell,r \in \mathcal{A}^*: |\ell|, |r| > 0, \\ \ell w r = v}} c^{|\ell| + |r|} \geq \\
\underbrace{2|\mathcal{A}| c^{-1} \sum_{n = 1}^{\infty} |\mathcal{F}_n| c^n}_{(a)} +
\underbrace{|\mathcal{A}| c^{-1} \sum_{v \in \mathcal{F}} \sum_{\substack{\ell \in \mathcal{A}^*: |\ell| < |v|, \\ \ell w \textrm{ begins with } v}} c^{|\ell|}}_{(b)} + \\
\underbrace{|\mathcal{A}| c^{-1} \sum_{v \in \mathcal{F}} \sum_{\substack{r \in \mathcal{A}^*: |r| < |v|, \\ w r \textrm{ begins with } v}} c^{|r|}}_{(c)} + 
\underbrace{2 c^{-2} \sum_{v \in \mathcal{F}} \sum_{\substack{\ell,r \in \mathcal{A}^*: |\ell|, |r| > 0, \\ \ell w r = v}} c^{|\ell| + |r|}}_{(d)}.
\end{multline*}
(The final inequality uses the fact that $c > |\mathcal{A}|^{-1}$.) We now claim that the portions (I) - (VII) of (\ref{part1}), (\ref{part2}), and (\ref{part3}) are all present in (a) - (d). 

Specifically, $(I) + (IV) = (a)$, $(II) \leq (b)$ and $(V) \leq (c)$,
and $(III) + (VI) + (VII) \leq (d)$ ($(VII)$ corresponds to the case $|\ell|, |r| > 1$ in $(d)$, $(III)$ to $|\ell| \geq 1$ and $|r| = 1$,
and $(VI)$ to $|\ell| = 1$ and $|r| \geq 1$; $|\ell| = |r| = 1$ is counted in both $(III)$ and $(VI)$, but this is taken care of by the 
factor of $2$ in $(d)$). This verifies (\ref{2millerineq}) and completes the proof.

\end{proof}

We are now prepared to prove that for `small' $\mathcal{F}$, $v,w \in G \Longrightarrow vw \in \mathcal{L}(X)$. 

\begin{theorem}\label{concat}
If there exists $c > |\mathcal{A}|^{-1}$ so that $\sum_{n = 1}^{\infty} |\mathcal{F}_n| c^n < \frac{|\mathcal{A}| c - 1}{2}$ and 
$\sum_{n = 1}^{\infty} n |\mathcal{F}_n| c^{n/3} < 1/4$, then for any $v,w \in G$, $vw \in \mathcal{L}(X)$.
\end{theorem}

\begin{proof}
This is essentially just an application of Theorem~\ref{2miller}. Assume that $v, w \in G$. Since $v$ and $w$ do not begin
or end with a heavy subword, any subword of a word $u \in \mathcal{F}$ contained in $vw$ has length at most $2|u|/3$. Therefore, for any $c$,
\begin{multline}\label{weightbd}
g_c(vw) = \sum_{t \in \mathcal{F}} \left[ \sum_{\substack{\ell \in \mathcal{A}^*: |\ell| < |t|, \\ \ell vw \textrm{ begins with } t}} c^{|\ell|} + 
\sum_{\substack{r \in \mathcal{A}^*: |r| < |t|, \\ vw r \textrm{ begins with } t}} c^{|r|} + 
2 \sum_{\substack{\ell,r \in \mathcal{A}^*: |\ell|, |r| > 0, \\ \ell vw r = t}} c^{|\ell| + |r|} \right] \\
\leq \sum_{t \in \mathcal{F}} 4tc^{t/3} = 4 \sum_{n = 1}^{\infty} n |\mathcal{F}_n| c^{n/3}.
\end{multline}
Therefore, $g_c(vw) < 1 = (|\mathcal{A}| c - 1)/2$ by assumption, and by Theorem~\ref{2miller}, $vw \in \mathcal{L}(X)$.

\end{proof}

We can now similarly define $G'$.

\begin{definition}\label{heavy'}
A \textbf{heavy$'$ subword} is a subword of some $w \in \mathcal{F}$ whose length is at least $|w|/4$. Denote by $\mathcal{H}'$ the set of heavy subwords.
\end{definition}

\begin{definition}
For any subshift $X$ with forbidden list $\mathcal{F}$, define $P'$ to be the set of all words in $\mathcal{L}(X)$ which do not have any heavy$'$ subword as a prefix, define $S'$ to be the set of all words in $\mathcal{L}(X)$ which do not have any heavy$'$ subword as a suffix, and define $G' = P' \cap S'$.
\end{definition}

We omit the proof of the following, since it is analogous to the proof of Theorem~\ref{concat}.

\begin{theorem}\label{concat'}
If there exists $c > |\mathcal{A}|^{-1}$ so that $\sum_{n = 1}^{\infty} |\mathcal{F}_n| c^n < \frac{|\mathcal{A}| c - 1}{2}$ and 
$\sum_{n = 1}^{\infty} n |\mathcal{F}_n| c^{n/4} < 1/4$, then for any $u,v,w \in G$, $uvw \in \mathcal{L}(X)$.
\end{theorem}

\section{Bounding $\mu(G_n)$ and/or $\mu(G'_n)$ away from $0$}\label{bounding}

The final hypothesis we must prove is that the sets $G_n$ and/or $G'_n$ have measures bounded away from $0$ along a syndetic set for any ergodic MME; we begin by working with $G_n$ via $P_n$ and $S_n$. Our proof will use some results from \cite{coded} about a class of subshifts called coded subshifts, which we now define.

\begin{definition}
Given a set $C \subseteq \mathcal{A}^*$ of \textbf{code words}, the \textbf{coded subshift} induced by $C$ is the closure of the set of shifts of biinfinite concatenations of words from $C$.
\end{definition}

Recall that $\mathcal{H}$ denotes the set of heavy subwords. 

\begin{theorem}\label{step1bd}
If there exists $\alpha < h(X)$ so that $\sum_{n = 1}^{\infty} |\mathcal{H}_n| e^{-n\alpha} < 1 - e^{-\alpha}$% and $\sum_{n} |\mathcal{F}_n| (2/|\mathcal{A}|)^n < \frac{1}{2}$
, then for every ergodic MME $\mu$ on $X$, $\mu(P_n)$ and $\mu(S_n)$ are bounded away from $0$.
\end{theorem}

\begin{proof}

Formally, define $\overline{P}_n = \{x \ : \ x([0, n)) \in P_n\}$; then $\mu(\overline{P}_n) = \mu(P_n)$, and $\overline{P}_1 \supset \overline{P}_2 \supset \ldots$. Therefore, we need only prove that $\mu(P_n) \nrightarrow 0$. For a contradiction, assume that $\mu(P_n) \rightarrow 0$. Choose any $\epsilon > 0$, and take $N$ s.t. $\mu(\overline{P}_N) < \epsilon$. For every $n$, define
\[
W_n := \{w \in \mathcal{L}_n(X) : |\{0 \leq i < n - N \ : \ w([i, i+N)) \in P_N\}| < 2n \epsilon\}.
\]
By Birkhoff's ergodic theorem (applied to $f = \chi_{\overline{P}_N}$), $\mu(W_n) \rightarrow 1$, and so by the Shannon-McMillan-Brieman theorem,
\begin{equation}\label{SMB}
\liminf_n \frac{\log |W_n|}{\log n} \rightarrow h(\mu) = h(X).
\end{equation} 
We will now give a procedure to bound $|W_n|$ from above. Choose any $w \in W_n$, begin at $i = 0$, and iterate the following procedure:
if $w([i, i+N)) \notin P_N$, then there exists $k \leq N$ so that $w([i, i + k))$ is a heavy subword; increment $i$ to $i + k$. 
If $w([i, i+N)) \in P_N$, then change $w(i)$ to a $*$ symbol and increment $i$ to $i+1$. When $i$ first becomes at least $n - N$,
replace $w(i)$, $\ldots$, $w(n - 1)$ with $*$ symbols, and end. Denote the resulting word by $c(w)$.

We note that $c(w)$ is a concatenation of heavy subwords of lengths less than $N$ and $*$s, and so $c(w)$ is contained in the language of the coded subshift with code word set $C_N = \{*\} \cup \bigcup_{i = 1}^N \mathcal{H}_i$; denote this coded subshift by $X_N$. We also note that there are fewer than $2n \epsilon + N$ $*$ symbols in $c(w)$ by definition of $W_n$. Finally, we note that every $w$ is determined entirely by $c(w)$ and the letters of $w$ at locations of $*$ symbols in $c(w)$, and so
\[
|W_n| \leq |\mathcal{A}|^{2n\epsilon + N} |\mathcal{L}_n(X_N)|.
\]
Taking logarithms, dividing by $n$, letting $n \rightarrow \infty$, and applying (\ref{SMB}) yields
\begin{equation}\label{wordbd}
h(X) \leq 2\epsilon \log |\mathcal{A}| + h(X_N).
\end{equation}

We will now use a result from \cite{coded}, which bounds the entropy of a coded subshift defined by code word set $C$ in terms of a generating function
\[
f_C(\alpha) := \sum_{w \in C} e^{-|w|\alpha}
\]
and the so-called limit subshift of biinfinite sequences obtained as limits of individual code words of increasing lengths. Since 
$C_N$ is finite, the corresponding limit subshift $L_N$ is empty. Theorem 1.7 of 
\cite{coded} states that for any $\alpha > h(L_N)$, if $f_{C_N}(\alpha) < 1$, then $h(C_N) \leq \alpha$. By definition, for the
$\alpha$ from the theorem,
\[
f_{C_N}(\alpha) = e^{-\alpha} + \sum_{n = 1}^{N} |\mathcal{H}_n| e^{-n\alpha} \leq e^{-\alpha} + \sum_{n = 1}^{\infty} |\mathcal{H}_n| e^{-n\alpha},
\]
which is less than $1$ by assumption. Therefore, $h(X_N) \leq \alpha$, and so by (\ref{wordbd}), $h(X) \leq 2\epsilon \log |\mathcal{A}| + \alpha$. Since
$\epsilon$ was arbitrary, $h(X) \leq \alpha$. This is a contradiction, implying that our original assumption was wrong and $\mu(P_n)$ is bounded away from $0$. 

The proof that $\mu(S_n)$ is bounded away from $0$ is nearly identical, using sets $\overline{S}_n = \{x \ : \ x((-n, 0]) \in S_n\}$ and representing 
words in an analogous version of $W_n$ as concatenations by starting from the right side and moving to the left.

%Finally, for each $N$, we can define the graph joining $\lambda_N$ between $\mu$ and $\phi_N(\mu)$, and then $\mu$ and $\phi_N(\mu)$ have  $d$-bar distance bounded from above by $1$ minus the measure of the union of cylinder sets of heavy subwords of lengths less than $N$. By assumption, this measure approaches $1$, and so $\phi_N(\mu)$ approaches $\mu$ in $d$-bar, implying that $h(\phi_N(\mu)) \rightarrow h(\mu)$. (CHECK THIS) Therefore, $h(\mu) \leq \log \alpha$, a contradiction, implying that our original assumption was wrong and $\mu(P_n)$ is bounded away from $0$. 

%For every $N$, define $\alpha_n$ to be the partition of $X$ (up to a $\mu$-null set) given by the cylinder sets of majority suffixes of length less than or equal to $N$ and the remainder $R_N$. By assumption, $\mu(R_N) \rightarrow 0$, and so $\alpha_n$ is a generating sequence of partitions (CHECK/DEFINE). Fix $N$, and we wish to bound from above the entropy $h(X, \alpha_N)$. For each $N$, $\mu$-a.e. $x \in X$ can be written as a concatenation of majority suffixes of length less than or equal to $N$ and visits to $R_N$. 

\end{proof}

\begin{theorem}\label{Gmeasbd}
If there exists $\alpha < h(X)$ so that $\sum_{n = 1}^{\infty} |\mathcal{H}_n| e^{-n\alpha} < 1 - e^{-\alpha}$, then for every ergodic MME $\mu$ on $X$, there exist $\epsilon > 0$ and a syndetic set $S$ so that $\mu(G_n) > \epsilon$ for $n \in S$.
\end{theorem}

\begin{proof}
Choose any ergodic MME $\mu$ on $X$, and as earlier, define $\overline{P}_n = \{x \ : \ x([0, n)) \in P_n\}$, $\overline{P} = \bigcap_{n = 1}^{\infty} \overline{P}_n$, and similarly define $\overline{S}_n = \{x \ : x([-n,0)) \in S_n\}$ and $\overline{S} = \bigcap_{n = 1}^{\infty} S_n$. By Theorem~\ref{step1bd}, there exists $\epsilon > 0$ so that $\mu(\overline{P}), \mu(\overline{S}) > \epsilon$. By ergodicity, $\mu \left( \bigcup_{i = -\infty}^{\infty} \sigma^i \overline{P} \right), \mu \left( \bigcup_{i = -\infty}^{\infty} \sigma^i \overline{S} \right) = 1$, and so there exists $N$ so that $\mu( \bigcup_{i = 0}^{N-1} \sigma^i \overline{P} ), \mu ( \bigcup_{i = 0}^{N-1} \sigma^i \overline{S} ) > 3/4$ (recall that $\mu$ is $\sigma$-invariant). 

Now, fix any $n$. Clearly
\[
\mu \left(\bigcup_{i = 0}^{N-1} \sigma^{-i} \overline{P} \cap \bigcup_{j = 0}^{N-1} \sigma^{-n+j} \overline{S}\right) > 1/2,
\]
and so there exist $0 \leq i,j < N$ so that $\mu(\sigma^{-i} \overline{P} \cap \sigma^{-n+j} \overline{S}) > \frac{1}{2N^2}$. Finally, we note that every $x$ in this intersection
has $x([i, n-j)) \in G_{n - j - i}$. Therefore, $\mu(G_{n - j - i}) > \frac{1}{2N^2}$. Since $n$ was arbitrary and $i + j < 2N$, we are finished by taking $\epsilon = \frac{1}{2N^2}$.

\end{proof}

The proof of the following result about $G'$ is extremely similar and so omitted.

\begin{theorem}\label{G'measbd}
If there exists $\alpha < h(X)$ so that $\sum_{n = 1}^{\infty} |\mathcal{H}'_n| e^{-n\alpha} < 1 - e^{-\alpha}$, then for every ergodic MME $\mu$ on $X$, there exist $\epsilon > 0$ and a syndetic set $S$ so that $\mu(G'_n) > \epsilon$ for $n \in S$.
\end{theorem}

Finally, we will want versions using $|\mathcal{F}_n|$ in the hypothesis (rather than $|\mathcal{H}_n|$ or $|\mathcal{H}'_n|$) for consistency with our other results. We again begin with $G$.

\begin{theorem}\label{Gmeascor}
If there exists $\alpha < h(X)$ so that $\sum_{n = 1}^{\infty} n^2 |\mathcal{F}_n| e^{-(n/3)\alpha} < 1 - e^{-\alpha}$, then for every ergodic MME $\mu$ on $X$, there exist $\epsilon > 0$ and a syndetic set $S$ so that $\mu(G_n) > \epsilon$ for all $n \in S$.
\end{theorem}

\begin{proof}

%We first note that the hypothesis implies that $h(X) > \log(3|\mathcal{A}|/5)$ by Theorem~\ref{entbound}. 
For each $w \in \mathcal{F}$ of length $n$, and each $\lceil n/3 \rceil \leq i \leq n$, there are less than or equal to $n - i + 1$ heavy subwords of $w$ of length $i$. Therefore, $|\mathcal{H}_i| \leq \sum_{n = i}^{3i} (n-i) |\mathcal{F}_n|$ and so
\begin{equation}\label{measbdeqn}
\sum_{i = 1}^{\infty} |\mathcal{H}_i| e^{-i\alpha} \leq \sum_{i = 1}^{\infty} e^{-i\alpha} \sum_{n = i}^{3i} (n-i+1) |\mathcal{F}_n|
\leq \sum_{n = 1}^{\infty} n^2 |\mathcal{F}_n| e^{-(n/3)\alpha} < 1 - e^{-\alpha}.
\end{equation}
The proof is now completed by using Theorem~\ref{Gmeasbd}. 

\end{proof}

The following is proved analogously. 

\begin{theorem}\label{G'meascor}
If there exists $\alpha < h(X)$ so that $\sum_{n = 1}^{\infty} n^2 |\mathcal{F}_n| e^{-(n/4)\alpha} < 1 - e^{-\alpha}$, then for every ergodic MME $\mu$ on $X$, there exist $\epsilon > 0$ and a syndetic set $S$ so that $\mu(G'_n) > \epsilon$ for all $n \in S$.
\end{theorem}

\section{Proving unique MME/K-property/Gibbs bounds for `small' $\mathcal{F}$}\label{final} 

We may now give a simple quantitative condition on $\mathcal{F}$ so that $X$ has a unique MME with the K-property.

We begin with a simple lemma yielding an explicit lower bound on $h(X)$.

\begin{theorem}\label{entbound}
If $\sum_{n=1}^{\infty} |\mathcal{F}_n| (3/|\mathcal{A}|)^n < \frac{1}{5}$, then $h(X) > \log (3|\mathcal{A}|/5)$.
\end{theorem}

\begin{proof}
This follows immediately from Theorem~\ref{millerent} with $k = 3|\mathcal{A}|/5 + 1$ and $c = 3/|\mathcal{A}|$.
\end{proof}

%\begin{theorem}\label{mainthm1}
%If $\sum_{n = 1}^{\infty} n^2 (3/|\mathcal{A}|)^{n/3} |\mathcal{F}_n| < \frac{1}{16}$, then $X$ has a unique MME.
%\end{theorem}

%\begin{proof}
%This is an immediate consequence of Corollary~\ref{Gmeascor} and Theorems~\ref{Lcombbd}, \ref{concat}, and \ref{mainthm0}.
%\end{proof}

%With slightly more effort, we can even prove that this unique MME is Kolmogorov.

\begin{theorem}\label{mainthm2}
If $\sum_{n = 1}^{\infty} n^2 |\mathcal{F}_n| (3/|\mathcal{A}|)^{n/3} < \frac{1}{36}$, then $X$ has a unique MME, which has the K-property.
\end{theorem}

\begin{proof}
Suppose that $\mathcal{F}$ and $\mathcal{A}$ satisfy the stated hypothesis, i.e. that
\begin{equation}\label{mainhyp}
\sum_{n = 1}^{\infty} n^2 |\mathcal{F}_n| (3/|\mathcal{A}|)^{n/3} < \frac{1}{36}.
\end{equation}

This clearly implies the hypothesis of Theorem~\ref{entbound}, and so $h(X) > \log(3|\mathcal{A}|/5)$. Therefore,
$2h(X) - \log |\mathcal{A}| > \log(9|\mathcal{A}|/25) > \log(|\mathcal{A}|/3)$, and so we can take $\beta = \log(|A|/3)$ in Theorem~\ref{Lcombbd}. The required inequality is
\[
\sum_{n = 1}^{\infty} n |\mathcal{F}_n| (3/|\mathcal{A}|)^n < \frac{1}{36},
\]
which follows from (\ref{mainhyp}). Therefore, the conclusion of Theorem~\ref{Lcombbd} holds. 

Now, define $G$ as was done in Section~\ref{defg}. We now wish to apply Theorem~\ref{concat} with $c = 3/|\mathcal{A}|$. The required inequalities are then
\[
\sum_{n = 1}^{\infty} |\mathcal{F}_n| (3/|\mathcal{A}|)^n < \frac{c|\mathcal{A}| - 1}{2} = 1 \textrm{ and}
\]
\[
\sum_{n = 1}^{\infty} n|\mathcal{F}_n| (3/|\mathcal{A}|)^{n/3} < \frac{1}{4},
\]
both of which clearly follow from (\ref{mainhyp}). Therefore, we know that $G$ has the concatenatability property from Theorem~\ref{concat}. 

Finally, we recall that $h(X) > \log(3|\mathcal{A}|/5)$, and so we can take $\alpha = \log(3|\mathcal{A}|/5)$ in Theorem~\ref{Gmeascor}. The required inequality is 
\[
\sum_{n = 1}^{\infty} n^2 |\mathcal{F}_n| (5/3|\mathcal{A}|)^{n/3} < 1 - \frac{5}{3|\mathcal{A}|}.
\]
This follows from (\ref{mainhyp}), if we note that $5/3 < 3$ and that $|\mathcal{A}| > 1$, and so the right-hand side above is at least $\frac{1}{6}$. So, the conclusion
of Theorem~\ref{Gmeascor} also holds, and now uniqueness of the MME $\mu$ on $X$ is an immediate consequence of Theorems~\ref{mainthm0}, \ref{Lcombbd}, \ref{concat}, and \ref{Gmeascor}.
 
We will now use a result of Ledrappier (Proposition 1.4 from \cite{ledrappier}) which states that if $X$ is a subshift and $X \times X$ has a unique MME (which must be $\mu \times \mu$), then $\mu$ has the K-property. We wish to use Theorem~\ref{mainthm0} on $X \times X$, and first note that if we define $G^{(2)}_n = G_n \times G_n$, then all properties from the hypotheses 
of Theorem~\ref{mainthm0} automatically carry over from $X$ except for the last one, that $\nu(G^{(2)}_n)$ is bounded away from $0$ along a syndetic set for every ergodic MME $\nu$ on $X \times X$. We will verify this by using Theorem~\ref{Gmeasbd}.

It is clear that $X \times X$ can be viewed as a subshift on alphabet $\mathcal{A}^2$ with forbidden list $\mathcal{F}^{(2)} = \{(v, w) \ : \ v \in \mathcal{F}, |w| = |v|\} \cup \{(v, w) \ : \ w \in \mathcal{F}, |v| = |w|\}$. By Theorem~\ref{entbound}, $h(X \times X) = 2h(X) \geq \log(9|\mathcal{A}|^2/25)$. Every heavy subword of a word in $\mathcal{F}^{(2)}$ has either first or second coordinate given by a heavy subword of $\mathcal{F}$, and so if we define $\mathcal{H}_n$ to be the set of heavy subwords of words in $\mathcal{F}$ and $\mathcal{H}^{(2)}_n$ to be the set of heavy subwords of words in $\mathcal{F}^{(2)}$,
then $|\mathcal{H}^{(2)}_n| \leq 2|\mathcal{H}_n| |\mathcal{A}|^n$, and as argued in the proof of Theorem~\ref{Gmeasbd}, $|\mathcal{H}_i| \leq \sum_{n = i}^{3i} (n-i+1) |\mathcal{F}_n|$ for all $i$. Then,
\begin{multline*}
\sum_{i = 1}^{\infty} |\mathcal{H}^{(2)}_i| (3/|\mathcal{A}|^2)^{i} \leq
\sum_{i = 1}^{\infty} 2|\mathcal{H}_i| (3/|\mathcal{A}|)^i \leq 2 \sum_{i = 1}^{\infty} (3/|\mathcal{A}|)^i \sum_{n = i}^{3i} (n-i+1) |\mathcal{F}_n| \\ 
\leq 2 \sum_{n = 1}^{\infty} n^2 |\mathcal{F}_n| (3/|\mathcal{A}|)^{n/3}.
\end{multline*}
We now apply Theorem~\ref{Gmeasbd} with $\alpha = \log(|\mathcal{A}|^2/3)$; note that $\alpha < h(X \times X)$, $e^{-\alpha} = \frac{3}{|\mathcal{A}|^2} \leq \frac{3}{4}$ and that by (\ref{mainhyp}), the final term above is less than $\frac{1}{18} < 1 - e^{-\alpha}$. Therefore, $\nu(G^{(2)}_n)$ is bounded away from $0$ along a syndetic set for every ergodic MME $\nu$ on $X \times X$. Now all hypotheses of Theorem~\ref{mainthm0} are satisfied, so $X \times X$ has a unique MME $\mu \times \mu$, implying that $\mu$ has the K-property.

\end{proof}

Under slightly stronger hypotheses, we can use Theorems~\ref{gibbsthm0} and \ref{easygibbs} to prove Gibbs bounds for the unique MME $\mu$ guaranteed by Theorem~\ref{mainthm2} for many words in 
$\mathcal{L}(X)$, in the spirit of similar results proved in \cite{CT0}. Recall the definition of $G'$ from Section~\ref{defg}.

\begin{theorem}\label{gibbsthm}
If $\sum_{n = 1}^{\infty} n^2 |\mathcal{F}_n| (3/|\mathcal{A}|)^{n/4} < \frac{1}{36}$ and $\mu$ is the unique MME on the induced subshift $X$ (guaranteed by Theorem~\ref{mainthm2}), then there exist $\epsilon > 0$ and a syndetic set $S$ so that $\mu(G'_n) > \epsilon$ for all $n \in S$, and there exist constants 
$D, D'$ so that for all $w \in G'$,
\[
De^{-|w|h(X)} \leq \mu([w]) \leq D'e^{-|w|h(X)}.
\]
\end{theorem}

\begin{proof}
Assume that $\mathcal{F}$ and $\mathcal{A}$ satisfy the hypotheses, i.e. that
\begin{equation}\label{mainhyp'}
\sum_{n = 1}^{\infty} n^2 |\mathcal{F}_n| (3/|\mathcal{A}|)^{n/4} < \frac{1}{36}.
\end{equation}

This is stronger than the hypothesis of Theorem~\ref{mainthm2}, and so we know that $h(X) > \log(3|\mathcal{A}|/5)$, that $X$ has a unique MME $\mu$, and that the conclusion of Theorem~\ref{Lcombbd} holds for $X$. 
Theorem~\ref{easygibbs} then implies the desired upper bound for $\mu$ (in fact, this upper bound holds for all words in $\mathcal{L}(X)$.) It remains only to verify the hypotheses of 
Theorem~\ref{gibbsthm0} to get the lower bound.

We now wish to apply Theorem~\ref{concat'} with $c = 3/|\mathcal{A}|$. The required inequalities are then
\[
\sum_{n = 1}^{\infty} |\mathcal{F}_n| (3/|\mathcal{A}|)^n < \frac{c|\mathcal{A}| - 1}{2} = 1 \textrm{ and}
\]
\[
\sum_{n = 1}^{\infty} n|\mathcal{F}_n| (3/|\mathcal{A}|)^{n/4} < \frac{1}{4},
\]
both of which clearly follow from (\ref{mainhyp'}). Therefore, we know that $G'$ has the threefold concatenatability property from Theorem~\ref{concat'}. 

Finally, we recall that $h(X) > \log(3|\mathcal{A}|/5)$, and so we can take $\alpha = \log(3|\mathcal{A}|/5)$ in Theorem~\ref{G'meascor}. The required inequality is 
\[
\sum_{n = 1}^{\infty} n^2 |\mathcal{F}_n| (5/3|\mathcal{A}|)^{n/4} < 1 - \frac{5}{3|\mathcal{A}|}.
\]
This follows from (\ref{mainhyp}), if we note that $5/3 < 3$ and that $|\mathcal{A}| > 1$, and so the right-hand side above is at least $\frac{1}{6}$. The conclusion
of Theorem~\ref{G'meascor} then holds, and our proof is now complete by using Theorem~\ref{gibbsthm0}.

\end{proof}

\begin{remark} 
We have made no effort to optimize constants throughout, instead just aiming to give a proof of concept for the heuristic ``infinite forbidden lists which are small in some way can imply good properties,'' and so some constants might be improvable. However, we would like to note that with our current proof, the $n/3$ in the exponent of Theorem~\ref{mainthm2} and $n/4$ in the exponent of Theorem~\ref{gibbsthm} cannot be improved. In other words, at the moment, to prove uniqueness of MME purely from the size of the forbidden list $\mathcal{F}$, we need exponential growth rate at most $O(|\mathcal{A}|^{1/3})$, and to prove the restricted lower Gibbs bound we need exponential growth rate at most $O(|\mathcal{A}|^{1/4})$. It is an interesting question whether these asymptotics are optimal (we strongly suspect not).
\end{remark}

\begin{remark}
It is natural to wonder whether these Gibbs bounds could have been used to give an alternate proof of uniqueness of MME in the same style
as proofs of Climenhaga and Thompson in \cite{CT0}, \cite{CT1}, and \cite{CT2}. Without giving full details, the technique there is to begin with
the definition of a known MME $\mu$ from the construction in the proof of Theorem~\ref{gibbsthm}, to prove $\mu$ is ergodic, and then to use the Gibbs property for $\mu$
to contradict mutual singularity of $\mu$ and another conjectured ergodic MME $\nu$. Interestingly, most of this proof could be adapted to our
setting, except for the proof of ergodicity of $\mu$! The proof used in the aforementioned papers uses some technical properties of the set on 
which the Gibbs bound holds (for us, $G'$) that do not necessarily hold in our setting, and we do not know of an alternate proof, so the bounds alone do not seem to clearly imply uniqueness in our setting.
\end{remark}

We note that there is no hope to prove a version of Theorem~\ref{gibbsthm} which holds on all of $\mathcal{L}(X)$. In particular, it is possible that $\mu$ is not even fully supported.

\begin{theorem}
There exists $X$ satisfying the hypotheses of Theorem~\ref{gibbsthm} whose unique MME $\mu$ is not fully supported.
\end{theorem}

\begin{proof}

Fix any alphabet $\mathcal{A}$ and any $N$, and choose any letter $a \in \mathcal{A}$. Define the forbidden list $\mathcal{F} = \{a^N b \ : \ b \in \mathcal{A} \setminus \{a\}\} \cup \{b a^N \ : \ b \in \mathcal{A} \setminus \{a\}\}$. Note that $a^N \in \mathcal{L}(X)$ since $a^{\mathbb{Z}} \in X$. However, it's clear from the definition of $\mathcal{F}$ that $[a^N] = \{a^{\mathbb{Z}}\}$.

We now simply note that $\mathcal{F}$ has exactly $2(|\mathcal{A}| - 1)$ words, all of length $N + 1$, and it's obvious that for large $\mathcal{A}$ and $N$, the hypotheses of Theorem~\ref{gibbsthm} are satisfied. However, for such $\mathcal{A}$ and $N$, the unique ergodic MME $\mu$ must have $\mu([a^N]) = 0$; since $\mu$ is ergodic, 
$\mu([a^N]) > 0$ would imply $\mu = \delta_{a^{\mathbb{Z}}}$, which is not possible since $\mu$ is an MME and $h(X) > 0$ by Theorem~\ref{entbound}.

\end{proof}

\begin{remark}
We note that the example $X$ from the previous theorem is not even topologically transitive, and so our hypotheses do not imply any sort of topological mixing property at all.
\end{remark}

%\section{A restricted Gibbs property}\label{gibbs}

\section{Applications}\label{apps}

Theorems~\ref{mainthm2} and \ref{gibbsthm} apply whenever $|\mathcal{F}_n|$ has growth rate much less than $|\mathcal{A}|^(n/4)$, and is small/empty for small values of $n$ (to make the value of the infinite series in our hypotheses small enough.) The following is just one simple explicit example of such a class.

\begin{corollary}\label{biga}
If $|\mathcal{A}| \geq 768$, $\mathcal{F}_n = \varnothing$ for $n < 14$, and $|\mathcal{F}_n| \leq 2^n$ for all $n \geq 13$, then $X$ satisfies the conclusions of Theorems~\ref{mainthm2} and \ref{gibbsthm}.
\end{corollary}

\begin{proof}

Under the given hypotheses,
\[
\sum_{n = 1}^{\infty} n^2 |\mathcal{F}_n| (3/|\mathcal{A}|)^{n/4} \leq \sum_{n = 14}^{\infty} n^2 256^{-n/4} 2^n = \sum_{n = 14}^{\infty} n^2 2^{-n} < 1/36.
\]
We now simply apply Theorems~\ref{mainthm2} and \ref{gibbsthm}.

\end{proof}

The rest of our applications are based on existing classes of subshifts in the literature, and each involves invoking properties of the subshift/forbidden list to more directly use our intermediate results in Sections~\ref{ub}-\ref{bounding} rather than directly applying Theorems~\ref{mainthm2} and \ref{gibbsthm}.

\subsection{$\alpha$-$\beta$ shifts}

We consider the $\alpha$-$\beta$ shifts (also called intermediate $\beta$-shifts) defined in \cite{hofbauerAB} and studied in, among other works, \cite{CT1} and \cite{liAB}. For all $\alpha$-$\beta$ shifts, the alphabet is $\mathcal{A} = \{0, \ldots, \ell\}$ for $\ell := \lceil \alpha + \beta \rceil - 1$, and we always take $\ell$ to be defined in this way.

The $\alpha$-$\beta$ shift is in fact a symbolic coding of the self-map $Tx := \alpha + \beta x \pmod 1$ of the unit interval $[0,1)$. Namely, define the intervals 
\[
J_0 = \Big[ 0, \frac{1 - \alpha}{\beta} \Big), 
J_1 = \Big[ \frac{1 - \alpha}{\beta}, \frac{2 - \alpha}{\beta} \Big), \ldots, J_{\ell} = \Big[ \frac{\ell - \alpha}{\beta}, 1 \Big),
\]
and for every $x \in [0,1)$ define a one-sided sequence $\Phi(x)$ where $T^n x \in J_{(\Phi(x))_n}$ for all $n \geq 0$. The one-sided $\alpha$-$\beta$ shift $X'_{\alpha, \beta}$ is then just $\overline{\{\Phi(x) \ : \ x \in [0,1)\}}$, and the two-sided $\alpha$-$\beta$ shift $X_{\alpha, \beta}$ is just the natural extension of $X'_{\alpha, \beta}$. 

If we define $\textbf{a} = \Phi(0)$ and $\textbf{b} = \lim_{t \rightarrow 1^-} \Phi(t)$ (suppressing dependence on $\alpha$ and $\beta$), then it was shown in \cite{hofbauerAB} that $X_{\alpha, \beta}$ can alternately be characterized as
\begin{equation}\label{abdef}
X_{\alpha, \beta} = \{x \in \mathcal{A}^{\mathbb{Z}} \ : \ \forall n, \textbf{a} \preceq x_n x_{n+1} x_{n+2} \ldots \preceq \textbf{b}\},
\end{equation}
where $\preceq$ denotes the lexicographic order.

It was recently shown in \cite{robAB} that when $\beta > 2$ and there is a prefix of $\alpha$ which never appears in $\beta$, then $X_{\alpha, \beta}$ has a unique measure of maximal entropy. Their proof technique uses the decomposition results of Climenhaga-Thompson from (\cite{CT1}). Using our results, we will be able to obtain entire intervals of $\alpha, \beta$ where 
$X_{\alpha, \beta}$ behaves much like a $\beta$-shift. The key observation is that by definition, all $\alpha$-$\beta$ shifts have very small forbidden lists.

\begin{lemma}\label{twowords}
For every $\alpha, \beta$, there exists a forbidden list $\mathcal{F}_{\alpha, \beta}$ inducing $X_{\alpha, \beta}$ where
$|(\mathcal{F}_{\alpha, \beta})_n| \leq 2\ell$ for all $n$, and $(\mathcal{F}_{\alpha, \beta})_n = \varnothing$ if $x(\alpha)$ begins with $n$ $0$s and $x(\beta)$ begins with $n$ $\ell$s.
\end{lemma}

\begin{proof}

Define $\mathcal{F}_{\alpha} = \{p a \ : \ p$ is a prefix of $\textbf{a}$ which is immediately followed by a letter less than $a\}$,
$\mathcal{F}_{\beta} = \{p b \ : \ p$ is a prefix of $\textbf{b}$ which is immediately followed by a letter greater than $b\}$,
and $\mathcal{F}_{\alpha, \beta} = \mathcal{F}_{\alpha} \cup \mathcal{F}_{\beta}$. The claimed bounds on 
$|(\mathcal{F}_{\alpha, \beta})_n|$ are immediate, and the lexicographic definition (\ref{abdef})
of $X_{\alpha, \beta}$ implies that $\mathcal{F}_{\alpha, \beta}$ is a forbidden list inducing $X_{\alpha, \beta}$. 

\end{proof}

We now fix the forbidden list from Lemma~\ref{twowords} as $\mathcal{F}_{\alpha, \beta}$ for every $\alpha, \beta$. Similarly, 
we define $\mathcal{H}_{\alpha, \beta}$ and $\mathcal{H}'_{\alpha, \beta}$ as the set of heavy/heavy$'$ subwords of $\mathcal{F}_{\alpha, \beta}$ and $G_{\alpha, \beta}$ and $G'_{\alpha, \beta}$ as the set of words which do not begin or end with a heavy/heavy$'$ subword of $\mathcal{F}_{\alpha, \beta}$.

We could already apply Theorems~\ref{mainthm2} and \ref{gibbsthm} when $\alpha + \beta > 3$ (since then $|\mathcal{A}| = \ell + 1 \geq 4$). However, with a little more work, we can treat all possible $\ell$ by instead
directly verifying the hypotheses of Theorems~\ref{mainthm0} and \ref{gibbsthm0} by using Theorems~\ref{Lcombbd}, \ref{concat}, \ref{concat'}, \ref{Gmeascor}, and \ref{G'meascor}. We first need a way to bound $h(X_{\alpha, \beta})$ from below.

%\begin{theorem}\label{easybeta}
%If $\alpha + \beta > 3$, $\alpha$ begins with $25$ $0$s, and $\beta$ begins with $25$ $\ell$s, then $X_{\alpha, \beta}$ satisfies the conclusions of Theorems~\ref{mainthm2} and \ref{gibbsthm} (with $G' = %G'_{\alpha, \beta}$).
%\end{theorem}

%\begin{proof}

%Since $\alpha + \beta > 3$, $|\mathcal{A}| = \ell + 1 \geq 4$. By Lemma~\ref{twowords}, $(\mathcal{F}_{\alpha, \beta})_n = \varnothing$ for $n \geq 25$. (CHECK) Therefore,
%\[
%\sum_{n = 1}^{\infty} n^2 |\mathcal{F}_n| (3/|\mathcal{A}|)^{n/4} \leq
%\sum_{n = 25}^{\infty} 4n^2 (3/4)^{n/4} < \frac{1}{36}.
%\]
%The proof is now completed by Theorem~\ref{gibbsthm}.

%\end{proof}

%For the more classical case where $\mathcal{A} = \{0,1\}$ (corresponding to $\alpha + \beta < 2$), we cannot directly apply our results, since they work only when $|\mathcal{A}| > 3$. However, with some effort, we can directly apply Theorem~\ref{mainthm0} and Theorem~\ref{gibbsthm0}. We first need a lemma bounding the entropy of $X_{\alpha, \beta}$ from below. 

\begin{lemma}\label{betaent}
If $\textbf{a}$ begins with $N$ $0$s and $\textbf{b}$ begins with $N$ $\ell$s, then $h(X_{\alpha, \beta}) \geq \frac{N-2}{N} \log (\ell + 1)$.
\end{lemma}

\begin{proof}
Fix $N$ and assume that $\textbf{a}$ begins with $N$ $0$s and $\textbf{b}$ begins with $N$ $\ell$s. Consider any sequence $x \in \{0, \ldots, \ell\}^{\mathbb{Z}}$ with the property that $x(Ni) = 0$ and $x(Ni + 1) = \ell$ for all $i \in \mathbb{Z}$. Then it is clear that every $N$-letter subword of $x$ contains both a $0$ and an $\ell$. By the assumed properties of $\textbf{a}$ and $\textbf{b}$, for every $n$, $x_n x_{n+1} x_{n+2} \ldots$ is lexicographically between $\textbf{a}$ and $\textbf{b}$, and so $x \in X_{\alpha, \beta}$. Since all such $x$ are in $X_{\alpha, \beta}$, $|\mathcal{L}_{Nk}(X_{\alpha, \beta})| \geq (\ell + 1)^{(N-2)k}$. Taking logarithms, dividing by $Nk$, and letting $k \rightarrow \infty$ yields the desired lower bound on $h(X_{\alpha, \beta})$.
\end{proof}

\begin{theorem}\label{hardbeta}
For every $\ell$, there exists $N$ %(COULD MAKE INDEP. OF $\ell$ BY WORKING HARDER) 
so that if $\textbf{a}$ begins with $N$ $0$s and $\textbf{b}$ begins with $N$ $\ell$s, then $X_{\alpha, \beta}$ satisfies the conclusions of Theorems~\ref{mainthm2} and \ref{gibbsthm} (with $G' = G'_{\alpha, \beta}$).
\end{theorem}

\begin{proof}
Fix any $\ell$. For brevity, we denote $X = X_{\alpha, \beta}$, $\mathcal{F} = \mathcal{F}_{\alpha, \beta}$, 
$\mathcal{H} = \mathcal{H}_{\alpha, \beta}$, $G = G_{\alpha, \beta}$, $\mathcal{H}' = \mathcal{H}'_{\alpha, \beta}$, and $G' = G'_{\alpha, \beta}$. By Lemma~\ref{twowords}, we know that $|\mathcal{F}_n| \leq 2\ell$ for all $n$ and $\mathcal{F}_n = \varnothing$ for $n \leq N$.
By Lemma~\ref{betaent}, we can force $h(X) \geq \frac{3}{5} \log (\ell + 1)$ by taking $N \geq 5$. Define $\beta_0 = 
\frac{1}{6} \log(\ell + 1) < 2h(X) - \log (\ell + 1)$, and define $c = \frac{2}{3}$. Note that $\beta_0 < h(X) - (\log(\ell + 1) - h(X)) \leq h(X)$ also. Now, note that all of the infinite series
\[
\sum_{n = 1}^{\infty} n |\mathcal{F}_n| e^{-n\beta_0}, \sum_{n = 1}^{\infty} n |\mathcal{F}_n| c^{n/3}, \sum_{n = 1}^{\infty} |\mathcal{F}_n| c^n, \sum_{n = 1}^{\infty} n^2 |\mathcal{F}_n| e^{-(n/3)\beta_0}
\]
converge, since $|\mathcal{F}_n|$ is bounded as $n \rightarrow \infty$. Then, it is clear that by taking $N$ large enough, we can force each of these series to be arbitrarily small (since each will become $0$ for $n \leq N$). Therefore, for large enough $N$, the hypotheses of Theorems~\ref{Lcombbd}, \ref{concat}, and \ref{Gmeascor} are satisfied (by taking $\beta = \beta_0$, $c = \frac{2}{3}$, and $\alpha = \beta_0$ respectively); note that as $N$ increases, the lower bound on $h(X)$ in Lemma~\ref{betaent} increases, and so the required fact that $\beta = \beta_0 < 2h(X) - \log |\mathcal{A}|$ (implying that
$\alpha = \beta_0 < h(X)$) will remain true. Existence of a unique MME $\mu$ now follows from Theorem~\ref{mainthm0}. 

The proof that $X \times X$ has the unique MME $\mu \times \mu$ is almost the same as in the proof of Theorem~\ref{mainthm2}. 
Again, all hypotheses of Theorem~\ref{mainthm0} carry over immediately for $X \times X$ except for the fact that
$\nu(G^{(2)}_n)$ is bounded away from $0$ (for $G^{(2)}_n = G_n \times G_n$) along a syndetic set for every ergodic MME $\nu$ on $X \times X$.

As before, $h(X \times X) = 2h(X)$, and if we take $\alpha_0 = \beta_0 + \log |\mathcal{A}|$, then $\alpha_0 < h(X \times X)$, and 
\begin{multline*}
\sum_{i = 1}^{\infty} |\mathcal{H}^{(2)}_i| e^{-i\alpha_0} \leq
\sum_{i = 1}^{\infty} 2|\mathcal{H}_i| |\mathcal{A}|^i e^{-i\alpha_0} = 
\sum_{i = 1}^{\infty} 2|\mathcal{H}_i| e^{-i\beta_0}\\
\leq 2 \sum_{i = 1}^{\infty} e^{-i\beta_0} \sum_{n = i}^{3i} (n-i+1) |\mathcal{F}_n| 
\leq 2 \sum_{n = 1}^{\infty} n^2 |\mathcal{F}_n| e^{-(i/3)\beta_0}.
\end{multline*}
As argued above, if $N$ is sufficiently large, this sum will be as small as desired (in particular less than $1 - e^{-\alpha_0}$), allowing for Theorem~\ref{Gmeascor} to be applied. Then, the K-property for $\mu$ via uniqueness of MME on $X \times X$ follows from Theorem~\ref{mainthm0}.

It remains only to prove the Gibbs bounds of Theorem~\ref{gibbsthm} for $\mu$ and $G'$. The infinite series
\[
\sum_{n = 1}^{\infty} |\mathcal{F}_n| c^n, \sum_{n = 1}^{\infty} n |\mathcal{F}_n| c^{n/4}, \sum_{n = 1}^{\infty} n^2 |\mathcal{F}_n| e^{-(n/4)\beta_0}
\]
again clearly converge (because $|\mathcal{F}_n| \leq 2\ell$), and so again we can make them arbitrarily small by taking $N$ large enough. This allows the use of Theorems~\ref{concat'} and \ref{G'meascor}, which implies the desired result by Theorems~\ref{gibbsthm0} and \ref{easygibbs}.

\end{proof}

The following corollary is almost immediate.

\begin{theorem}\label{hardbetacor}
For every $\ell$, there exists $\epsilon > 0$ so that if $\alpha < \epsilon$, $\beta > \ell - \epsilon$, and $\alpha + \beta < \ell$, then $X_{\alpha, \beta}$ satisfies the conclusions of Theorems~\ref{mainthm2} and \ref{gibbsthm} (with $G' = G'_{\alpha, \beta}$).
\end{theorem}

\begin{proof}
This is an immediate consequence of Theorem~\ref{hardbeta}, given the observation that by definition of $\textbf{a}$ and $\textbf{b}$, taking $\epsilon$ very small will force $\textbf{a}$ and $\textbf{b}$ to begin with arbitrarily many $0$s and $\ell$s respectively.
\end{proof}

\begin{remark} 
We note that Hofbauer proved in \cite{hofbauerPM} that any piecewise increasing map on the interval (including $T: x \mapsto \alpha + \beta x \pmod 1$) which is topologically transitive is also intrinsically ergodic.
It is also known (see \cite{glend}) that binary $\alpha$-$\beta$ shifts of entropy at least $\sqrt{2}$ are always topologically transitive, which means that the intrinsic ergodicity portion of Theorem~\ref{hardbeta} is not new for $\ell = 1$ (and likely for larger $\ell$ as well, though we do not have a reference).

However, we note that one can consider the more general class of $a$-$b$ shifts (see \cite{ganAB}) defined via lexicographic lower/upper bounds $a$/$b$ in the sense of (\ref{abdef}). Theorem~\ref{hardbeta} holds with no changes for this class, and it is not hard to check that there exist $a$-$b$ shifts satisfying the hypotheses of Theorems~\ref{hardbeta} and \ref{hardbetacor} which are not transitive. For instance, for any $\ell$ and $n$, one can define 
\[
a = .0^n \ell^n 0^n \ell^n \ldots, b = .\ell^n 0^n \ell^n 0^n \ldots,
\]
and in this $a$-$b$ shift, the only sequence with $0^n$ after the decimal point is $a$ itself, which clearly precludes topological transitivity. Therefore, for $a$-$b$ shifts, our results can yield intrinsic ergodicity for a different reason than the results of \cite{hofbauerPM}. 
\end{remark}

\subsection{Bounded density shifts}

We first recall the definition of the bounded density subshifts of \cite{stanley}. 

\begin{definition}\label{bdd}
For any $\mathcal{A}_k = \{0, 1, \ldots, k\}$ and increasing subadditive $h: \mathbb{N} \rightarrow \mathbb{N}$, the \textbf{bounded density shift} associated to $k$ and $h$, denoted $X_{k,h}$, is the subshift defined by $\mathcal{F}_{k,h} = \bigcup_{n = 1}^{\infty} \{w \in \mathcal{A}^n \ : \ w_1 + \ldots + w_n > h(n)\}$.
\end{definition}

We also define a slightly more general version of bounded density shifts where the sum of the letters in a word is bounded from above and below.

\begin{definition}\label{bdd+}
For any $\mathcal{A}^{\pm}_k = \{-k, \ldots, -1, 0, 1, \ldots, k\}$ and increasing subadditive $h: \mathbb{N} \rightarrow \mathbb{N}$, the 
\textbf{signed bounded density shift} associated to $k$ and $h$, denoted $X^{\pm}_{k,h}$, is the subshift defined by 
$\mathcal{F}^{\pm}_{k,h} = \bigcup_{n = 1}^{\infty} \{w \in \mathcal{A}^n \ : \ w_1 + \ldots + w_n > h(n)\} \cup 
\bigcup_{n = 1}^{\infty} \{w \in \mathcal{A}^n \ : \ w_1 + \ldots + w_n < -h(n)\}$.
\end{definition}

When $h$ is a constant function, $\mathcal{A}^{\pm}_k$ is quite similar to the so-called charge-constrained shift (see Example 1.2.7 in \cite{LM} for more details). By subadditivity, for any (signed) bounded density shift, $h(n)/n$ approaches some constant $\alpha$ (called the gradient), and $h(n) \geq n\alpha$ for all $n$. 

A useful property of bounded density shifts is that they are hereditary (defined in \cite{kerrli}), meaning that whenever a letter in a point in the shift is replaced by a smaller letter, the resulting point is still in the shift. Hereditary shifts are known to have many useful properties (see \cite{PR}, \cite{dom1}, \cite{weisshered}, and \cite{dom2}), and so it's not surprising that they often have unique MME. However, signed bounded density shifts have no obvious hereditary properties. Nevertheless, we will be able to apply our results to them as well. We will need the following simple lemma.

\begin{lemma}\label{zeroout}
Any $w$ on $\mathcal{A}_k$ (or $\mathcal{A}^{\pm}_k$) which does not contain a word from the forbidden list $\mathcal{F}_{k,h}$ (or $\mathcal{F}^{\pm}_{k,h}$)
is in the language of $X_{k,h}$ (or $X^{\pm}_{k,h}$).
\end{lemma}

\begin{proof}
We will show for any such $w$, the sequence $0^{\infty} w 0^{\infty}$ is in the relevant subshift. To see this, note that every subword of $0^{\infty} w 0^{\infty}$ of length $n$ consists of a subword of $w$ of some length $m$ (which has sum with absolute value less than or equal to $h(m)$ by assumption) and some $0$s, and so its sum has absolute value less than or equal to $h(m)$ as well, which must be less than or equal to $h(n)$ since $h$ is increasing. No subword of $0^{\infty} w 0^{\infty}$ is in the relevant forbidden list, and so $0^{\infty} w 0^{\infty}$ is in the relevant subshift, implying that $w$ is in the language and completing the proof.
\end{proof} 

To our knowledge, nothing regarding uniqueness of MME or properties of a unique MME has been proved for any bounded density shifts except when $X_{k,h}$ satisfies some specification (or weakened specification) property, which implies uniqueness by the main results of \cite{bowen} or \cite{gapspec}. This requires $h(n) - n\alpha$ to have sublogarithmic growth, and so treats a very narrow class of possible $k$ and $h$. 

In contrast, we will be able to prove these properties whenever $h(n)$ is above a certain constant times $kn$. When $h(n)$ is close to $kn$, the forbidden list is fairly `small' (since very few words of length $n$ have sum with absolute value near $kn$), and so we could apply Theorems~\ref{mainthm2} and \ref{gibbsthm} immediately to derive uniqueness/K-property/Gibbs bounds. However, it turns out we can directly prove relevant properties about $G_{k,h}$ (or $G^{\pm}_{k,h}$) in this setting, and so using Theorems~\ref{mainthm0}, \ref{gibbsthm0}, and \ref{easygibbs} will be significantly more efficient. 

We define the sets $\mathcal{H}_{k,h}$ (and $\mathcal{H}^{\pm}_{k,h}$) of heavy subwords of $\mathcal{F}_{k, h}$ (and $\mathcal{F}^{\pm}_{k,h}$) as in Section~\ref{defg}, i.e. subwords of words in $\mathcal{F}_{k,h}$ of length at least one-third that of the forbidden word. Define $G_{k,h}$ (and $G^{\pm}_{k,h}$) as in Section~\ref{defg}, i.e. words which do not begin or end with a heavy subword. We will not need a separate definition for $G'$ here using one-fourth of forbidden words, because we will be able to use Lemma~\ref{zeroout} to show that these sets already have the stronger threefold concatenation property usable for Theorem~\ref{gibbsthm0}.

\begin{theorem}\label{bddthm}
If $k > 9e$, $h(n) = nk$ for $n < 11$ and $h(n) > nk(1 - \frac{1}{9e})$ for all $n$, then $X_{k,h}$ satisfies the conclusions of Theorems~\ref{mainthm2} and \ref{gibbsthm}
(with $G' = G_{k,h}$).
\end{theorem}

\begin{proof}

Define $B = \frac{1}{9e}$. The reader may check that this implies that, for $k > 9e$,
\begin{equation}\label{bfact}
\frac{e(1+3kB)(k+1)}{k^2(1-B)^2} < \frac{1}{2}.
\end{equation}

For convenience, we will abbreviate $\mathcal{A} = \mathcal{A}_k$, $X = X_{k,h}$, $\mathcal{F} = \mathcal{F}_{k,h}$, $\mathcal{H} = \mathcal{H}_{k,h}$, and $G = G_{k,h}$. We must verify the hypotheses of Theorems~\ref{mainthm0}, \ref{gibbsthm0}, and \ref{easygibbs}. First, we note for future reference that by definition, the full shift on $\{0, \ldots, 1, \lfloor k(1 - B) \rfloor\}$ is contained in $X$ by definition, and therefore
\begin{equation}\label{bdentbd}
h(X) \geq \log(1 + \lfloor k(1-B) \rfloor) > \log(k(1-B)).
\end{equation}

Next, note that for $n < 11$, $\mathcal{F}_n = \varnothing$ by definition. For $n \geq 11$, by definition, 
(recall that $h(n) > nk(1 - B)$),
\[
\mathcal{F}_n = \{w \in \mathcal{A}^n \ : \ w_1 + \ldots + w_n > h(n)\} \subset 
\{w \in \mathcal{A}^n \ : \ w_1 + \ldots + w_n > nk(1 - B)\}.
\]

Therefore,
\begin{multline}\label{bddFbd}
|\mathcal{F}_n| \leq 
|\{w \in \mathcal{A}^n \ : \ w_1 + \ldots + w_n > nk(1 - B)\}| \\
= |\{w \in \mathcal{A}^n \ : \ w_1 + \ldots + w_n < nkB\}| \leq {n(1 + kB) \choose n} \leq (e(1 + kB))^n.
\end{multline}
(Here, the equality is due to the obvious permutation on $\mathcal{A}^n$ given by subtracting letters from $k$ coordinatewise, the first inequality is due to the fact that nonnegative integer partitions of positive integers less than or equal to $nkB$ are in bijective correspondence with ways of placing $n$ ``dividing lines'' in a row of $nkB$ objects, and the second inequality is a well-known upper bound on binomial coefficients.)

By (\ref{bdentbd}), $2h(X) - \log |A| > 2\log (k(1-B)) - \log (k+1) = \log \left( \frac{k^2(1-B)^2}{k+1} \right)$, and we can take $\beta = \log \left( \frac{k^2(1-B)^2}{k+1} \right)$ for Theorem~\ref{Lcombbd}. The relevant infinite series is then 
\begin{multline*}
\sum_{n = 1}^{\infty} n |\mathcal{F}_n| e^{-\beta n} \leq
\sum_{n = 11}^{\infty} n \left(\frac{(e(1+kB)(k+1)}{k^2(1-B)^2}\right)^n \\ <
\sum_{n = 11}^{\infty} n \left(\frac{(e(1+3kB)(k+1)}{k^2(1-B)^2}\right)^n <
\sum_{n = 11}^{\infty} n 2^{-n} < 1/36.
\end{multline*}
(Here, the first inequality uses (\ref{bddFbd}) and the final inequality uses (\ref{bfact}).) Therefore, $X$ satisfies the conclusion of Theorem~\ref{Lcombbd}.

By definition, any concatenation $uvw$ of $u,v,w \in G$ cannot contain any word from $\mathcal{F}$, and so by Lemma~\ref{zeroout}, all such words are in $\mathcal{L}(X)$, verifying both specification properties for $G$ required in Theorems~\ref{mainthm0} and \ref{gibbsthm0} (for $R = 0$).

We now estimate the sizes of $\mathcal{H}_n$ in order to apply Theorem~\ref{Gmeasbd}. Since every forbidden word $w$ of length $n$ satisfies $w_1 + \ldots w_n > h(n) \Longrightarrow w_1 + \ldots + w_n > n\alpha > nk(1 - B)$, it is easily checked that every heavy subword $u$ of $w$ with length $k \geq n/3$ satisfies $u_1 \ldots + u_k > nk(1 - 3B)$. Since $\mathcal{F}_n = \varnothing$ for $n < 11$, by definition $\mathcal{H}_n = \varnothing$ for $n < 4$. For $n \geq 4$, an estimate virtually identical to that of (\ref{bddFbd}) shows that
\[
|\mathcal{H}_n| \leq (e(1 + 3kB))^n.
\]

Now, by (\ref{bdentbd}), we may take $\alpha = \log(k(1 - B))$ in Theorem~\ref{Gmeasbd}; the relevant infinite series is
\begin{multline*}
\sum_{n = 1}^{\infty} |\mathcal{H}_n| e^{-\alpha n} \leq
\sum_{n = 4}^{\infty} \left(\frac{e(1+3kB)}{k(1-B)}\right)^n \\ <
\sum_{n = 4}^{\infty} \left(\frac{e(1+3kB)(k+1)}{k^2(1-B)^2}\right)^n < 
\sum_{n = 4}^{\infty} (1/2)^{n} = \frac{1}{8}.
\end{multline*}
(Again, the final inequality comes from (\ref{bfact}).) Since $1 - e^{-\alpha} = 1 - \frac{1}{k(1-B)} > 1/8$, $G$ satisfies the conclusion of Theorem~\ref{Gmeasbd}. Now, uniqueness of the MME $\mu$ and upper and lower Gibbs bounds on $G$ follow from Theorems~\ref{mainthm0}, \ref{gibbsthm0}, and \ref{easygibbs}.

It remains only to show that $\mu$ has the K-property, which we do via the same proof technique as Theorem~\ref{mainthm2}, i.e. showing that $\mu \times \mu$ is the unique MME on 
$X \times X$ by Theorem~\ref{mainthm0} and applying Proposition 1.4 from \cite{ledrappier}.

As in that proof, we define $G^{(2)}_n = G_n \times G_n$, $\mathcal{F}^{(2)} = \{(v, w) \ : \ v \in \mathcal{F}, |w| = |v|\} \cup \{(v, w) \ : \ w \in \mathcal{F}, |v| = |w|\}$, and $\mathcal{H}^{(2)}$ to be the set of heavy subwords of words in $\mathcal{F}^{(2)}$. Also as in that proof, all required hypotheses for Theorem~\ref{mainthm0} follow immediately except for the fact that $\nu(G^{(2)}_n)$ is bounded away from $0$ along a syndetic set for every MME $\nu$ on $X \times X$. Since $h(X \times X) = 2h(X) > 2\log(k(1-B))$ by (\ref{bdentbd}), we can take $\alpha' = 2\log(k(1-B))$ in Theorem~\ref{Gmeasbd}; the relevant infinite series is
\begin{multline*}
\sum_{n = 1}^{\infty} |\mathcal{H}^{(2)}_n| e^{-\alpha' n} \leq
\sum_{n = 4}^{\infty} 2|\mathcal{H}_n| |\mathcal{A}|^n \left(\frac{1}{k^2(1-B)^2}\right)^n \\ \leq
\sum_{n = 4}^{\infty} 2 \left(\frac{e(1+3kB)(k+1)}{k^2(1-B)^2}\right)^n <
2 \sum_{n = 4}^{\infty} (1/2)^n = \frac{1}{4}.
\end{multline*}
(The final inequality here uses (\ref{bfact}).) Since $1 - e^{-\alpha'} = 1 - \frac{1}{k^2(1-B)^2} > 1/4$, $G^{(2)}$ satisfies the conclusion of Theorem~\ref{Gmeasbd}. Now, Theorem~\ref{mainthm0} shows that $\mu \times \mu$ is the unique MME on $X \times X$, and so that $\mu$ has the K-property, completing the proof.

\end{proof}

The same result extends to $X^{\pm}_{k, h}$ with few changes.

\begin{theorem}\label{bddpmthm}
If $k > 9e$, $h(n) = nk$ for $n < 11$ and $h(n) > nk(1 - \frac{1}{9e})$ for all $n$, then $X^{\pm}_{k,h}$ satisfies the conclusions of Theorems~\ref{mainthm2} and \ref{gibbsthm} (with $G' = G^{\pm}_{k,h}$).
\end{theorem}

\begin{proof}
Define $B = \frac{1}{9e}$. The reader may check that this implies that, for $k > 9e$,
\begin{equation}\label{bfact2}
\frac{e(1+3kB)(2k+1)}{(2k-2)^2(1-B)^2} < \frac{1}{2}.
\end{equation}

The rest of the proof is very similar to that of Theorem~\ref{bddthm}, and we will just summarize the changes. Firstly,
by Lemma~\ref{zeroout}, $G'$ has the specification properties in the hypotheses of Theorems~\ref{mainthm0} and \ref{gibbsthm0} (for $R = 0$)exactly as before. Secondly, the full shift on 
$\{-\lfloor k(1 - B) \rfloor, \ldots, 0, \ldots, \lfloor k(1 - B) \rfloor\}$ is contained in $X = X^{\pm}_{k, h}$ by definition, and therefore
\begin{equation}\label{bdentbd2}
h(X) \geq \log(1 + 2 \lfloor k(1-B) \rfloor) > \log(2k(1-B) - 1) \geq \log((2k-2)(1-B))
\end{equation}
(since $B < 1/2$). Our alphabet $\mathcal{A} = \mathcal{A}^{\pm}_{k, h}$ has size $2k + 1$, and our bounds for 
$\mathcal{F} = \mathcal{F}^{\pm}_{k, h}$ are correspondingly slightly larger. Specifically,
\begin{multline}\label{Fpmbd}
|\mathcal{F}_n| \leq \
|\{w \in \mathcal{A}^n \ : \ |w_1 + \ldots + w_n| > nk(1 - B)\}| \\ 
= 2|\{w \in \mathcal{A}^n \ : \ w_1 + \ldots + w_n > nk(1 - B)\}| = 
2|\{w \in [0, 2k]^n \ : \ w_1 + \ldots + w_n < nkB\}| \\
\leq 2 {n(1 + kB) \choose n} \leq 2(e(1 + kB))^n.
\end{multline}
(Here, the first equality is due to the obvious permutation on $\mathcal{A}^n$ given by flipping signs coordinatewise and 
the second equality is due to the bijection from $\mathcal{A}^n$ to $[0,2k]^n$ given by subtracting each letter from $k$.)

A nearly identical argument along the same lines from that of Theorem~\ref{bddthm} shows that for $\mathcal{H} = \mathcal{H}^{\pm}_{k,h}$,
\[
|\mathcal{H}_n| \leq 2(e(1 + 3kB))^n.
\]

So, all $|\mathcal{F}_n|$ and $|\mathcal{H}_n|$ have upper bounds on cardinality twice what was used in the proof of Theorem~\ref{bddthm}. The reader may check that if we use
$\beta = \log \left( \frac{(2k-2)^2(1-B)^2}{2k+1} \right)$ for Theorem~\ref{Lcombbd}, $\alpha = \log((2k-2)(1-B))$ for the use of Theorem~\ref{Gmeasbd} for $G = G^{\pm}_{k, h}$, and
$\alpha' = \log ((2k-2)^2(1-B)^2)$ for the use of Theorem~\ref{Gmeasbd} for $G^{(2)} = G \times G$, then all relevant infinite series end up with exponential term
less than $\frac{e(1+3kB)(2k+1)}{(2k-2)^2(1-B)^2}$. Therefore, by (\ref{bfact2}), they satisfy nearly the same final upper bounds, but are all doubled due to the
doubled upper bounds on cardinality for $|\mathcal{F}_n|$ and $|\mathcal{H}_n|$. However, all of the relevant inequalities in Theorem~\ref{bddthm} 
still hold when the infinite series are doubled, and so the argument used there will go through to prove the desired properties for $X^{\pm}_{k,h}$ as well.

\end{proof}

\begin{remark}
We presented Theorems~\ref{bddthm} and \ref{bddpmthm} both to give an example of a class of subshifts with exponentially growing
$|\mathcal{F}_n|$ which can be treated with our results, and also to demonstrate that if the forbidden list $\mathcal{F}$ is homogeneous in some sense (so that $\mathcal{H}_n$ is not too much larger than $\mathcal{F}_n$), then it can be possible to prove uniqueness of MME/Kolmogorov/Gibbs even when $\mathcal{F}_n$ is only bounded from above by $(|\mathcal{A}|/E)^n$ for some constant $E$ (as in (\ref{bddFbd}) and (\ref{Fpmbd})), rather than the $(|\mathcal{A}|/E)^{n/3}$ required for Theorem~\ref{mainthm2} or the
$(|\mathcal{A}|/E)^{n/4}$ required for Theorem~\ref{gibbsthm}.

Again, we have not attempted here to optimize relevant constants, and imagine that they could be improved at the expense of some more careful estimates.

\end{remark}

\bibliographystyle{plain}
\bibliography{nearsfts}

\end{document}